\documentclass[reqno,12pt]{amsart}
\usepackage{amsmath,amssymb,latexsym,soul,cite,mathrsfs}
\usepackage{xcolor,enumitem,graphicx}
\usepackage[colorlinks=true,urlcolor=blue,
citecolor=red,linkcolor=blue,linktocpage,pdfpagelabels,
bookmarksnumbered,bookmarksopen]{hyperref}
\usepackage[english]{babel}
\usepackage[a4paper,margin=1in]{geometry}

\usepackage[T5,T1]{fontenc}

\newtheorem{theoremletter}{Theorem}

\newtheorem{theo}{Theorem}[section]

\newtheorem{remark}{Remark}[section]

\newtheorem{lemma}{Lemma}[section]

\newtheorem{prop}{Proposition}[section]

\newtheorem{cor}{Corollary}[section]

\numberwithin{equation}{section}

\title[P\'olya--Szeg\"o on submanifolds]{P\'olya--Szeg\"o Inequality on Submanifolds of Riemannian Manifolds with Nonnegative Curvature and Applications}

\author[A. Do]{Anh Xuan Do}
\author[N. Lam]{Nguyen Lam}
\author[G. Lu]{Guozhen Lu}
\author[R. Ponciano]{Raon\'{\i} Ponciano}

\address[Anh Xuan Do]{Dep. Mathematics, University of Connecticut
	\newline\indent
	06269, Storrs-CT, United States of America}
\email{\href{mailto:anh.do@uconn.edu}{anh.do@uconn.edu}}

\address[Nguyen Lam]{School of Science and the Environment, Memorial University of Newfoundland
\newline\indent
A2H5G4, Corner Brook-NL, Canada}
\email{\href{mailto:nlam@mun.ca}{nlam@mun.ca}}

\address[Guozhen Lu]{Dep. Mathematics, University of Connecticut
	\newline\indent
	06269, Storrs-CT, United States of America}
\email{\href{mailto:guozhen.lu@uconn.edu}{guozhen.lu@uconn.edu}}

\address[Raon\'i Ponciano]{Department of Mathematics,
	Federal University of ABC
	\newline\indent
	09280-560, Santo Andr\'e-SP, Brazil}
\email{\href{mailto:raoni.ponciano@ufabc.edu.br}{raoni.ponciano@ufabc.edu.br}}


\begin{document}
	\begin{abstract}		
    We prove a P\'olya--Szeg\"o inequality for functions defined on an $n$-dimensional submanifold $\Sigma$ of a complete noncompact Riemannian manifold with nonnegative sectional curvature. The associated rearrangement is a Schwarz rearrangement on $\mathbb R^n$, and the constant depends on the $L^n$-norm of the mean curvature of $\Sigma$ and an isoperimetric quantity obtained by Brendle. As applications, we derive Sobolev, Log-Sobolev, Hardy, and Gagliardo--Nirenberg inequalities on submanifolds of arbitrary codimension under a small total mean curvature assumption. In the critical Sobolev case, we obtain Moser--Trudinger inequalities on finite-volume submanifolds and exact growth inequalities on submanifolds with infinite volume. Under suitable assumptions, the P\'olya--Szeg\"o constant equals one; in this case, the critical constants in the inequalities coincide with the sharp Euclidean ones.
\end{abstract}
	\maketitle	
	\begin{center}
		\footnotesize
		\tableofcontents
	\end{center}

\section{Introduction}

The P\'olya--Szeg\"o inequality is one of the fundamental tools connecting symmetrization, the isoperimetric problem, and functional inequalities. In its classical Euclidean form, decreasing rearrangement does not increase the Dirichlet integral, and this fact provides a sharp Sobolev-type estimate; see the monograph of P\'olya and Szeg\"o \cite{MR43486} and Talenti's rearrangement proof of the sharp Sobolev inequality \cite{MR601601}. This has become a standard tool in the analysis of elliptic equations, geometric inequalities, and Sobolev-type embeddings. The present paper aims to extend this mechanism to functions on submanifolds by applying Brendle's Sobolev inequality of Michael-Simon type for submanifolds of manifolds with nonnegative sectional curvature \cite{MR4612577, MS73}.

Roughly, our main arguments for our proofs are as follows. We first apply Brendle's inequality to obtain an isoperimetric inequality on the submanifold $\Sigma$; see Proposition~\ref{propIP}. In this inequality, the isoperimetric constant depends on $\|H\|_{L^n(\Sigma)}$, the $L^n$-norm of the mean curvature vector. This estimate is then used in the coarea-based proof of the P\'olya--Szeg\"o inequality, yielding a P\'olya--Szeg\"o principle for submanifolds; see Theorem~\ref{MainPS}. The resulting estimate compares the tangential gradient on $\Sigma$ with the Euclidean gradient of the Schwarz rearrangement in $\mathbb R^n$. This comparison allows several sharp Euclidean inequalities on $\mathbb R^n$ to be transferred to $\Sigma$, up to the same explicit P\'olya--Szeg\"o constant. Related rearrangement results have been established on manifolds with nonnegative Ricci curvature \cite[Proposition 3.1]{MR4566705}, Carnot Groups \cite{MR3961339}, model manifolds \cite{muratori2025concentrationcomparisonnonlineardiffusion}, $CD(K,N)$ spaces \cite{MR4088507,MR4970248,MR3608721}, and Euclidean submanifolds \cite{MR4977133}.

The main contributions of this paper are: (i) we establish the P\'olya–Szeg\"o inequality for submanifolds of an arbitrary complete noncompact Riemannian manifold with nonnegative sectional curvature, with an explicit constant $C_{\mathrm{PS}}$; and (ii) we derive several applications showing how Euclidean inequalities can be transported into inequalities on such submanifolds. Specifically, to highlight the many applications our theorem admits, we prove Sobolev, log-Sobolev, Gagliardo–Nirenberg and  Hardy inequalities, critical and subcritical Moser-Trudinger inequalities and their analogues  with exact growth, etc.

In this paper, we study symmetrization of functions defined on a submanifold $\Sigma$ of a Riemannian manifold $(M,g)$. Throughout, $M$ denotes a manifold of dimension $n+m$, while $\Sigma\subset M$ is an $n$-dimensional submanifold. Consequently, $n$ always denotes the dimension of $\Sigma$, and $m$ its codimension; throughout the paper we assume $n\geq2$ and $m\geq1$. We also denote $H$ as the mean curvature vector of $\Sigma$. Moreover, $\Sigma$ is endowed with the natural measure $\mathcal H^n_g$, namely the $n$-dimensional Hausdorff measure associated with the metric $g$. In the particular case where $\Sigma$ is a submanifold of the Euclidean space $(\mathbb R^n,\delta)$, we write $\mathcal H^n_\delta$ for the corresponding Hausdorff measure, where $\delta=\delta_{ij}\mathrm dx^i\otimes\mathrm dx^j$ is the Euclidean Riemannian metric. When $m=0$, this coincides with the Lebesgue measure, that is, $\mathcal L^n=\mathcal H^n_\delta$. We denote by $|B^n|:=\mathcal L^n(B^n)$ the $n$-dimensional volume of the unit ball $B^n\subset\mathbb R^n$ and $\omega_{n-1}:=\mathcal H^{n-1}_\delta(\mathbb S^{n-1})$ the $(n-1)$-dimensional volume of the unit sphere $\mathbb S^{n-1}\subset\mathbb R^{n}$. 

Let $M$ be a complete, noncompact manifold of dimension $n+m$ with nonnegative Ricci curvature. The asymptotic volume ratio of $M$ is defined by
\begin{equation*}
\mathrm{AVR}_g:=\lim _{r \rightarrow \infty} \frac{\mathcal H_g^{n+m}(\{p \in M: d_M(p,q)<r\})}{\left|B^{n+m}\right| r^{n+m}},
\end{equation*}
where $d_M(\cdot,\cdot)$ denotes the distance in $M$ induced by the Riemannian metric $g$ and $q\in M$ is a fixed point (the limit does not depend on the choice of $q$). By the Bishop-Gromov volume comparison theorem, one has $\mathrm{AVR}_g\leq1$.

The following theorem was established by Brendle. It was originally formulated only for $m\geq2$, but the analogous result holds for $m=1$. This is because we can always increase the codimension by considering another manifold $\widetilde\Sigma=\Sigma\times\{0\}\subset M\times\mathbb R=\widetilde M$, which does not affect the value of $\mathrm{AVR}_g$. Indeed, writing $\widetilde g=g\oplus\mathrm dt^2$ for the product metric on $\widetilde M=M\times\mathbb R$, geodesic balls in $\widetilde M$ decompose as $\mathcal H^{n+2}_{\widetilde g}(B_r^{\widetilde M})=\int_{-r}^r\mathcal H^{n+1}_g(B^M_{\sqrt{r^2-t^2}})\mathrm dt$ and since $|B^{n+2}|=|B^{n+1}|\int_{-1}^1(1-s^2)^{\frac{n+1}{2}}\mathrm ds$, dividing by $r^{n+2}$ and letting $r\to\infty$ gives $\mathrm{AVR}_{\widetilde g}=\mathrm{AVR}_g$. Note also that $\widetilde\Sigma$ has the same mean curvature vector as $\Sigma$, and that $\widetilde M$ still has nonnegative sectional curvature.

\begin{theoremletter}[See {\cite[Theorem 1.4]{MR4612577}}]\label{theobrendle}
Let $M$ be a complete noncompact manifold of dimension $n+m$ with nonnegative sectional curvature. Let $\Sigma$ be a compact submanifold of $M$ of dimension $n\geq2$ (possibly with boundary $\partial\Sigma$), and let $u$ be a positive smooth function on $\Sigma$. If $m \geq 1$, then
\begin{equation*}
\int_{\Sigma} \sqrt{\left|\nabla^{\Sigma} u\right|^{2}+u^{2}|H|^{2}}\mathrm d\mathcal H^n_g+\int_{\partial \Sigma} u\mathrm d\mathcal H^{n-1}_g \geq C_{\mathrm{Bre}}\left(\int_{\Sigma} u^{\frac{n}{n-1}}\mathrm d\mathcal H^n_g\right)^{\frac{n-1}{n}},
\end{equation*}
where $C_{\mathrm{Bre}}$ denotes the Brendle's constant given by
\begin{equation*}
C_{\mathrm{Bre}}=C_{\mathrm{Bre}}\left(n,m,\mathrm{AVR}_g^{\frac1n}\right)=\left\{\begin{array}{ll}
n\mathrm{AVR}_g^{\frac1n}|B^n|^{\frac{1}{n}}&\mbox{if }m=1,2,\\
n\mathrm{AVR}_g^{\frac{1}{n}}\left(\frac{(n+m)\left|B^{n+m}\right|}{m\left|B^{m}\right|}\right)^{\frac{1}{n}}&\mbox{if }m\geq3.\end{array}\right.
\end{equation*}
\end{theoremletter}

As a consequence of the previous inequality, we first establish an isoperimetric inequality for subsets $\Omega\subset\Sigma$; see Proposition~\ref{propIP}. More precisely, we prove that
\begin{equation*}
\left(\mathcal H^n_g(\Omega)\right)^{\frac{n-1}n}\leq C\mathcal H^{n-1}_g(\partial\Omega)
\end{equation*}
This estimate is then used to prove our first main result: a P\'olya--Szeg\"o inequality for submanifolds $\Sigma\subset M$. We emphasize that, in our setting, the Schwarz symmetrization is defined on $\mathbb R^n$, rather than on $\Sigma$ itself. The precise definition of this symmetrization is given in Section~\ref{sec2}.
\begin{theo}[P\'olya--Szeg\"o inequality for submanifolds]
Let $M$ be a complete noncompact Riemannian manifold of dimension $n+m$ with nonnegative sectional curvature and $m\geq1$. Let $\Sigma$ be a submanifold of $M$ of dimension $n\geq2$ (possibly with boundary) satisfying $\|H\|_{L^n(\Sigma)}<C_{\mathrm{Bre}}$. If $u\in W^{1,p}_0(\Sigma)$ for some $1\leq p<\infty$ and $u^*\colon\mathbb R^n\to[0,\infty)$ is its Schwarz rearrangement, then $u^*\in W^{1,p}(\mathbb R^n;[0,\infty))$ and\label{MainPS}

\begin{equation*}
\|\nabla u^*\|_{L^p(\mathbb R^n)}\leq C_{\mathrm{PS}}\|\nabla^\Sigma u\|_{L^p(\Sigma)},
\end{equation*}
where $C_{\mathrm{PS}}$ denotes the P\'olya--Szeg\"o constant given by
\begin{equation}\label{CPS}
C_{\mathrm{PS}}\!=\!C_{\mathrm{PS}}(n,m,\mathrm{AVR}_g,\|H\|_{L^n(\Sigma)})=\!\left\{\begin{array}{ll}
\!\dfrac{n|B^n|^{\frac1n}}{n\mathrm{AVR}_g^{\frac1n}|B^n|^{\frac1n}-\|H\|_{L^n(\Sigma)}}&\mbox{if }m=1,2,\\
\!\dfrac{n|B^n|^{\frac1n}}{n\mathrm{AVR}_g^{\frac{1}{n}}\left(\frac{(n+m)|B^{n+m}|}{m|B^m|}\right)^{\frac1n}-\|H\|_{L^n(\Sigma)}}&\mbox{if }m\geq3.
\end{array}\right.
\end{equation}
\end{theo}

\begin{remark}
Assuming that $u$ attains equality in the P\'olya--Szeg\"o inequality, then its level sets necessarily attain equality in the corresponding isoperimetric inequality. In the case of codimension one or two, this implies that $M=\mathbb R^{n+m}$ and $\Sigma$ is a flat ball, possibly with infinite radius. This restriction on the codimension is required, since the characterization of equality cases in the isoperimetric inequality is currently available only in this setting, as established in \cite[Theorem 1.6]{MR4612577}. For more details, we refer to Remark~\ref{remarkequality}.
\end{remark}

\begin{remark}
Let us analyze the assumption $\|H\|_{L^n(\Sigma)}<C_{\mathrm{Bre}}$. If this condition is removed from Theorem \ref{MainPS}, the same arguments used in the proof yields
\begin{equation*}
\|\nabla^\Sigma u\|_{L^p(\Sigma)}^p\geq\left(\dfrac{C_{\mathrm{Bre}}}{n|B^n|^{\frac1n}}\right)^p\int_{S(u)}\int_{\{u^*=s\}}|\nabla u^*|^{p-1}\mathrm d\mathcal H^{n-1}_\delta\left[1-\dfrac{n|B^n|^{\frac1n}\int_{\{u>s\}}|H|\mathrm d\mathcal H^n_g}{C_{\mathrm{Bre}}\mathcal H^{n-1}_\delta(\{u^*=s\})}\right]^p\mathrm ds,
\end{equation*}
where $S(u)$ denotes the set of all $s\in\mathbb R$ such that
\begin{equation*}
n|B^n|^{\frac{1}{n}}\int_{\{u>s\}}|H|\mathrm d\mathcal H^n_g<C_{\mathrm{Bre}}\mathcal H^{n-1}_\delta(\{u^*=s\}).
\end{equation*}
Note that if the mean curvature is sufficiently large, it may happen that $S(u)=\emptyset$. In that case, the right-hand side vanishes, and the inequality reduces to the trivial one: $\|\nabla^\Sigma u\|_{L^p(\Sigma)}\geq0$. Thus, without some assumption on $H$, the argument does not produce any useful inequality.
\end{remark}

In our P\'olya--Szeg\"o inequality, a constant $C_{\mathrm{PS}}$ appears, whereas in the Euclidean setting this constant is equal to $1$. However, the following proposition characterizes the situations in which $C_{\mathrm{PS}}=1$.

\begin{prop}\label{propCPS}
Assume the hypotheses of Theorem \ref{MainPS}. Then the constant $C_{\mathrm{PS}}$ satisfies
\begin{equation*}
C_{\mathrm{PS}}\geq1.
\end{equation*}
Moreover, equality holds if and only if
\begin{equation*}
\mathrm{AVR}_g=1,\ m=1\mbox{ or }2,\mbox{ and }\Sigma\mbox{ is minimal}.
\end{equation*}
\end{prop}

We now present our applications for the P\'olya--Szeg\"o inequality. Throughout, $u$ is defined on the submanifold $\Sigma$, and $C_{\mathrm{PS}}$ denotes the P\'olya--Szeg\"o constant, which can be taken equal to 1 under the additional assumptions of Proposition \ref{propCPS}. Our first application is the following Sobolev inequality.

\begin{cor}[Sobolev inequality on submanifolds]\label{corSob}
Let $M$ be a complete noncompact manifold of dimension $n+m$ with nonnegative sectional curvature. Let $\Sigma$ be a submanifold of $M$ of dimension $n$ (possibly with boundary) satisfying $\|H\|_{L^n(\Sigma)}<C_{\mathrm{Bre}}$. Let $1<p<n$, and let $p^*:=\frac{np}{n-p}$ be its Sobolev conjugate. Then, for all $u\in W^{1,p}_0(\Sigma)$, we have
\begin{equation}\label{SobSubmfd}
    \|u\|_{L^{p^*}(\Sigma)} \leq C_{\mathrm{Sob}}\|\nabla^{\Sigma}u\|_{L^p(\Sigma)},
\end{equation}
where $C_{\mathrm{Sob}}=TA(n,p)C_{\mathrm{PS}}$, with $TA(n,p)$ being the Talenti's sharp constant \cite{MR0463908} given by
$$TA(n,p)=\dfrac{1}{\sqrt{\pi}n^{\frac1p}}\left(\dfrac{p-1}{n-p}\right)^{1-\frac{1}{p}}\left(\dfrac{\Gamma (1+\frac{n}{2})\Gamma(n)}{\Gamma(\frac{n}{p})\Gamma(1+n-\frac{n}{p})}\right)^{\frac{1}{n}}.$$
\end{cor}

The same argument also gives the following log-Sobolev inequality; see, for instance, the classical works of Gross \cite{MR420249}, Weissler \cite{MR479373}, as well as the works by  Del Pino-Dolbeault \cite{MR1957678}  and Lam-Lu \cite{LamLu-CKN-ANS}.
\begin{cor}[Log-Sobolev inequality on submanifolds]\label{corLog}
Let $M$ be a complete noncompact manifold of dimension $n+m$ with nonnegative sectional curvature. Let $\Sigma$ be a complete submanifold of $M$ of dimension $n$ satisfying $\|H\|_{L^n(\Sigma)}<C_{\mathrm{Bre}}$. Then, for $1<p<n$, we have, for all $u\in W^{1,p}_0(\Sigma)$ with $\|u\|_{L^p(\Sigma)}=1$,
    $$\int_{\Sigma} |u|^p\log|u|\mathrm d\mathcal{H}^n_g\leq \dfrac{n}{p^2}\ln\left(C_{\mathrm{LS}}\int_{\Sigma}|\nabla^\Sigma u|^p\mathrm d\mathcal{H}^n_g\right),$$
where $C_{\mathrm{LS}}=\mathcal{L}_{n,p}C_{\mathrm{PS}}^p$ with
    \begin{equation*}
    \mathcal{L}_{n,p} = \frac{p}{n}\pi^{-\frac{p}{2}} \left( \frac{p-1}{e} \right)^{p-1} \left( \frac{\Gamma\left( \frac{n}{2} + 1 \right) }{\Gamma\left( \frac{n(p-1)}{p} + 1 \right)} \right)^{\frac{p}{n}}
    \end{equation*}
    being the Euclidean log-Sobolev constant.
\end{cor}

Next, we establish a version of the Gagliardo--Nirenberg inequality in our setting. We refer to \cite{MR1940370} for the Euclidean case.

\begin{cor}[Gagliardo--Nirenberg inequality on submanifolds]\label{cor13}
Let $M$ be a complete noncompact manifold of dimension $n+m$ with nonnegative sectional curvature. Let $\Sigma$ be a submanifold of $M$ of dimension $n$ (possibly with boundary $\partial\Sigma$) satisfying $\|H\|_{L^n(\Sigma)}<C_{\mathrm{Bre}}$. Assume that $1<p<n$ and $1<q<\frac{p(n-1)}{n-p}$. Then the following statements hold.

\smallskip

\noindent (i) If $q>p$, then for every $u\in C^\infty_0(\Sigma)$,
\begin{equation*}
\|u\|_{L^{p\frac{q-1}{p-1}}(\Sigma)}\leq C_{\mathrm{GN1}}\|\nabla^\Sigma u\|_{L^p(\Sigma)}^\theta\|u\|^{1-\theta}_{L^q(\Sigma)},
\end{equation*}
where $\theta=\frac{(q-p)n}{(q-1)(np+pq-nq)}$ and  $C_{\mathrm{GN1}}=\mathrm{GN}_1C_{\mathrm{PS}}^\theta$ with
\begin{equation*}
\mathrm{GN}_1\!=\!\left(\dfrac{q-p}{p\sqrt\pi}\right)^\theta\!\left(\dfrac{pq}{n(q-p)}\right)^{\frac{\theta}{p}}\!\left(\dfrac{np+pq-nq}{pq}\right)^{\frac{p-1}{p(q-1)}}\!\left(\dfrac{\Gamma\left(q\frac{p-1}{q-p}\right)\Gamma\left(\frac{n}{2}+1\right)}{\Gamma\left(\frac{p-1}{p}\frac{np+pq-nq}{q-p}\right)\!\Gamma\left(n\frac{p-1}{p}+1\right)}\!\right)^{\frac{\theta}{n}}\!.
\end{equation*}

\smallskip

\noindent (ii) If $q<p$, then for every $u\in C^\infty_0(\Sigma)$,
\begin{equation*}
\|u\|_{L^q(\Sigma)}\leq C_{\mathrm{GN2}}\|\nabla^\Sigma u\|^\theta_{L^p(\Sigma)}\|u\|^{1-\theta}_{L^{p\frac{q-1}{p-1}}(\Sigma)},
\end{equation*}
where $\theta=\frac{(p-q)n}{q(np+pq-nq-p)}$ and $C_{\mathrm{GN2}}=\mathrm{GN}_2C_{\mathrm{PS}}^\theta$ with
\begin{equation*}
\mathrm{GN}_2\!=\!\left(\!\dfrac{p-q}{p\sqrt\pi}\!\right)^\theta\!\left(\!\dfrac{pq}{n(p-q)}\!\right)^{\frac{\theta}{p}}\!\left(\!\dfrac{pq}{np+pq-nq}\!\right)^{\frac{1-\theta}{p}\frac{p-1}{q-1}}\!\left(\!\dfrac{\Gamma\left(\frac{p-1}{p}\frac{np+pq-nq}{p-q}+1\right)\Gamma\left(\frac{n}{2}+1\right)}{\Gamma\left(q\frac{p-1}{p-q}+1\right)\Gamma\left(n\frac{p-1}{p}+1\right)}\!\right)^{\frac{\theta}{n}}\!.
\end{equation*}
\end{cor}

We also obtain a Hardy inequality which is an extension of the classical result of Carron  on $L^2$-Hardy inequalities  \cite{Carron97}, and $L^p$ case of Kombe and Özaydin \cite{KO09} to submanifolds; see Remark \ref{remarkhardy} for a discussion of the sharp constant. 

\begin{theo}[Hardy inequality on submanifolds]\label{theohardy}
Let $M$ be a complete noncompact manifold of dimension $n+m$ with nonnegative sectional curvature. Let $\Sigma$ be a complete submanifold of $M$ of dimension $n$ without boundary satisfying $\|H\|_{L^n(\Sigma)}<C_{\mathrm{Bre}}$ and $\text{Ric}_{\Sigma}\geq 0$. Then, for all $1<p<n$ and $u \in W^{1,p}_0(\Sigma)$, 
$$\|\nabla^\Sigma u\|_{L^p(\Sigma)}^p \geq \left(\dfrac{n-p}{p C_{\mathrm{PS}}}\right)^p \int_\Sigma \dfrac{|u|^p}{d(x,o)^p}\mathrm d\mathcal{H}_g^n,$$
where $o$ is a fixed point in $\Sigma$ and $d(\cdot,\cdot)$ denotes the distance in $\Sigma$.
\end{theo}

Assuming $\mathcal H^n_g(\Sigma)<\infty$, then Corollary \ref{corSob} yields the embedding $W^{1,n}_0(\Sigma)\hookrightarrow L^q(\Sigma)$ for all $1\leq q<\infty$. Motivated by the classical works of Trudinger \cite{MR216286} and Moser \cite{zbMATH03323360}, we obtain the following Moser--Trudinger inequality.

\begin{theo}[Moser--Trudinger inequality on submanifolds]\label{theoMoser}
    Let $M$ be a complete noncompact manifold of dimension $n+m$ with nonnegative sectional curvature. Let $\Sigma$ be a submanifold of $M$ of dimension $n$ (possibly with boundary) satisfying $\|H\|_{L^n(\Sigma)}<C_{\mathrm{Bre}}$ and $\mathcal H_g^n(\Sigma)<\infty$. Then, there exists a constant $C_0>0$ such that
    $$\sup_{u\in \mathcal{A}} \int_{\Sigma} e^{\alpha |u|^{\frac{n}{n-1}}}\mathrm d\mathcal H^n_g\leq C_0 \mathcal H_g^n(\Sigma) ,$$
    for all $\alpha \leq \alpha_\Sigma$, where $\mathcal{A}:=\{u \in C^\infty_0(\Sigma): \|\nabla^{\Sigma} u\|_{L^n(\Sigma)}\leq 1 \},$
    and 
    $$\alpha_\Sigma=\dfrac{n\omega_{n-1}^{\frac{1}{n-1}}}{C_{\mathrm{PS}}^{\frac{n}{n-1}}},$$
    with $\omega_{n-1}$ the $(n-1)$-volume of the unit sphere $\mathbb{S}^{n-1}\subset\mathbb R^n$.
\end{theo}

\begin{remark}
The two hypotheses of Theorem \ref{theoMoser}, namely $\|H\|_{L^n(\Sigma)}<C_{\mathrm{Bre}}$ and $\mathcal H^n_g(\Sigma)<\infty$, are compatible in general. For example, they hold if $\Sigma$ is compact with $\partial\Sigma\neq\emptyset$ and small total mean curvature. However, if $\Sigma$ is additionally complete with $\partial\Sigma=\emptyset$, the two conditions become incompatible. Indeed, fix $o\in\Sigma$ and let $V(R):=\mathcal H^n_g(B^\Sigma_R(o))$. By completeness, $\overline{B^\Sigma_R(o)}$ is compact. Then, we may apply Proposition \ref{propIP} to $\Omega=\overline{B^\Sigma_R(o)}$, together with the coarea formula $V'(R)=\mathcal H^{n-1}_g(\partial B^\Sigma_R(o))$, to get $V(R)^{\frac{n-1}n}\leq C_{\mathrm{IP}}V'(R)$ a.e., i.e.\ $\frac{\mathrm d}{\mathrm dR}V(R)^{1/n}\geq(nC_{\mathrm{IP}})^{-1}$. Hence $V(R)\to\infty$ as $R\to\infty$. Thus, $\mathcal H^n_g(\Sigma)$ is infinite if $\Sigma$ is noncompact. On the other hand, if $\Sigma$ is compact, Corollary \ref{cor319} gives a contradiction. Therefore, a complete manifold $\Sigma$ without boundary cannot satisfy both hypotheses simultaneously.
\end{remark}

In the supercritical case $\alpha>n\omega_{n-1}^{\frac{1}{n-1}}$, the supremum is infinite. This is established in the next result. Although this theorem is already known (see \cite[Proposition 3.6]{MR1236762} for higher-order derivatives), we include it for completeness.
\begin{theo}\label{theoIM}
Let $\Sigma$ be an $n$-dimensional Riemannian manifold with $\mathcal H^n_g(\Sigma)<\infty$. Then
\begin{equation*}
\sup_{u\in\mathcal A}\int_\Sigma e^{\alpha|u|^{\frac{n}{n-1}}}\mathrm d\mathcal H^n_g=\infty,
\end{equation*}
for all $\alpha>n\omega_{n-1}^{\frac{1}{n-1}}$, where $\mathcal{A}:=\{u \in C^\infty_0(\Sigma): \|\nabla^{\Sigma} u\|_{L^n(\Sigma)}\leq 1 \}$.
\end{theo}
Our results do not address the intermediate range
\begin{equation}\label{rangeconj}
\frac{n\omega_{n-1}^{\frac{1}{n-1}}}{C_{\mathrm{PS}}^{\frac{n}{n-1}}}<\alpha\leq n\omega_{n-1}^{\frac{1}{n-1}},
\end{equation}
which remains open. However, under the conditions with $C_{\mathrm{PS}}=1$ (given in Proposition \ref{propCPS}), there is no gap and the results are sharp.

When $\mathcal H^n_g(\Sigma)$ is infinite, the embedding $W^{1,n}_0(\Sigma)\hookrightarrow L^q(\Sigma)$ does not hold for $1\leq q<n$, and the exponential must be modified. We define
\begin{equation*}
\exp_n(t)=e^t-\sum_{i=0}^{n-2}\dfrac{t^i}{i!},\quad\forall n\in\mathbb N.
\end{equation*}
  For the theory of Moser-Trudinger-Adams inequalities with exact growth on the Euclidean space, see for example, \cite{MR3336837, MR3355498, LamLu-EG, LTZ-ANS, MR3848068, MR4721763}, and \cite{LuTang-JGA} on hyperbolic spaces. In our setting, we prove:
\begin{theo}[Exact growth inequality on submanifolds]\label{theoEG}
Let $M$ be a complete noncompact Riemannian manifold of dimension $n+m$ with nonnegative sectional curvature. Let $\Sigma$ be a submanifold of $M$ of dimension $n$ (possibly with boundary) satisfying $\|H\|_{L^n(\Sigma)}<C_{\mathrm{Bre}}$. There exists a constant $C>0$ such that for all $u\in W^{1,n}_0(\Sigma)$ satisfying $\|\nabla^\Sigma u\|_{L^n(\Sigma)}\leq1$,
\begin{equation*}
\int_\Sigma\dfrac{\exp_n\left(\alpha_\Sigma|u|^{\frac{n}{n-1}}\right)}{(1+|u|)^{\frac{n}{n-1}}}\mathrm d\mathcal H^n_g\leq C\|u\|^n_{L^n(\Sigma)},
\end{equation*}
where
\begin{equation*}
\alpha_\Sigma=\dfrac{n\omega_{n-1}^{\frac{1}{n-1}}}{C_{\mathrm{PS}}^{\frac{n}{n-1}}}.
\end{equation*}
\end{theo}

As consequences, we derive analogs of the works by Adachi--Tanaka \cite{MR1646323}, do \'O \cite{MR1704875},  Li--Ruf \cite{MR2400264} and Lam-Lu \cite{LamLu-JDE}.

\begin{cor}\label{cor1}
Under the assumptions of Theorem \ref{theoEG}, for any $\alpha<\alpha_\Sigma$, there exists $C_\alpha>0$ such that
\begin{equation*}
\int_\Sigma\exp_n(\alpha|u|^{\frac{n}{n-1}})\mathrm d\mathcal H^n_g\leq C_\alpha\|u\|^n_{L^n(\Sigma)}.
\end{equation*}
\end{cor}
\begin{cor}\label{cor2}
Under the assumptions of Theorem \ref{theoEG}, for any $\tau>0$, there exists a constant $C_\tau>0$ such that
\begin{equation*}
\sup_{\underset{\|\nabla^\Sigma u\|^n_{L^n(\Sigma)}+\tau\|u\|^n_{L^n(\Sigma)}\leq1}{u\in W^{1,n}_0(\Sigma)}}\int_\Sigma\exp_n(\alpha_\Sigma|u|^{\frac{n}{n-1}})\mathrm d\mathcal H^n_g\leq C_\tau.
\end{equation*}
\end{cor}

Finally, we show that the inequalities fail for $\alpha>n\omega_{n-1}^{\frac{1}{n-1}}$ or exponents $q<\frac{n}{n-1}$. As in the Moser--Trudinger case, the behavior in the range \eqref{rangeconj} is also an open question.

\begin{theo}\label{theosuperc}
Let $\Sigma$ be an $n$-dimensional Riemannian manifold. Then there exists a sequence $(\psi_\ell)$ such that $\|\nabla^\Sigma\psi_\ell\|_{L^n(\Sigma)}\leq 1$ and
\begin{equation}\label{eqsuc1}
    \dfrac{1}{\|\psi_\ell\|^n_{L^n(\Sigma)}}\int_{\Sigma}\dfrac{\exp_n\left(n\omega_{n-1}^{\frac{1}{n-1}}|\psi_\ell|^{\frac{n}{n-1}}\right)}{(1+|\psi_\ell|)^q}\mathrm d\mathcal H^n_g\overset{\ell\to\infty}\longrightarrow \infty,\quad\forall q<\frac{n}{n-1}.
\end{equation}
Moreover, for all $\alpha>n\omega_{n-1}^{\frac{1}{n-1}}$ and $q\geq0$, we have
\begin{equation}\label{eqsuc2}
\int_{\Sigma}\dfrac{\exp_n\left(\alpha|\psi_\ell|^{\frac{n}{n-1}}\right)}{(1+|\psi_\ell|)^q}\mathrm d\mathcal H^n_g\overset{\ell\to\infty}\longrightarrow \infty.
\end{equation}
\end{theo}

\smallskip

\noindent\textbf{Outline of the paper:} The rest of the paper is divided as follows. In Section \ref{sec2}, we establish an isoperimetric inequality for $\Omega\subset\Sigma$ (see Proposition \ref{propIP}), introduce the Schwarz symmetrization, prove the P\'olya--Szeg\"o inequality in Theorem \ref{MainPS}, and show Proposition \ref{propCPS} concerning the constant $C_{\mathrm{PS}}$. Section \ref{sec3} is devoted to the Sobolev, Log-Sobolev, Gagliardo--Nirenberg, and Hardy inequalities, providing the proofs of Corollaries \ref{corSob}, \ref{corLog}, and \ref{cor13}, and Theorem \ref{theohardy}. Section \ref{sec4} focuses on the Moser--Trudinger inequality, proving Theorems \ref{theoMoser} and \ref{theoIM}. Finally, Section \ref{sec5} is dedicated to the exact growth inequalities, establishing Theorems \ref{theoEG} and \ref{theosuperc} as well as Corollaries \ref{cor1} and \ref{cor2}.

\section{The P\'olya--Szeg\"o Inequality}\label{sec2}

Now we state the isoperimetric inequality in our setting. However, the characterization of the equality case has been established only in codimension two (and, consequently, also in codimension one); see \cite[Theorem 1.6]{MR4612577}.
\begin{prop}\label{propIP}
Let $M$ be a complete noncompact manifold of dimension $n+m$ with nonnegative sectional curvature. Let $\Sigma\subset M$ be an $n$-dimensional submanifold, possibly with boundary, and let $\Omega\subset\Sigma$ be a compact $n$-dimensional submanifold, possibly with boundary $\partial \Omega$ as a submanifold. If $\|H\|_{L^n(\Sigma)}<C_{\mathrm{Bre}}$, then
\begin{equation*}
\left(\mathcal H^n_g(\Omega)\right)^{\frac{n-1}n}\leq C_{\mathrm{IP}}\mathcal H^{n-1}_g(\partial\Omega),
\end{equation*}
where the isoperimetric constant is given by
\begin{equation*}
C_{\mathrm{IP}}=C_{\mathrm{IP}}\left(n,m,\mathrm{AVR}_g^{\frac1n},\|H\|_{L^n(\Sigma)}\right)=\dfrac{1}{C_{\mathrm{Bre}}-\|H\|_{L^n(\Sigma)}}.
\end{equation*}
Moreover, when $m=1$ or $m=2$, the equality holds if $\Omega$ is a flat round $n$-ball and $M$ is isometric to $\mathbb{R}^{n+m}$. Conversely, the equality case implies this geometric characterization when $m=1$ or $m=2$.
\end{prop}
\begin{proof}
This proof is based on \cite[Proposition 6]{MR4977133}. By applying $u\equiv1$ in Theorem \ref{theobrendle} to $\Omega\subset\Sigma$ (and noting that the mean curvature of $\Omega$ as a submanifold of $M$ coincides with $H|_\Omega$ because $\Omega$ is open in $\Sigma$), we obtain
\begin{equation*}
C_{\mathrm{Bre}}(\mathcal H^n_g(\Omega))^{\frac{n-1}n}\leq \mathcal H^{n-1}_g(\partial\Omega)+\int_\Omega|H|\mathrm d\mathcal H^n_g.
\end{equation*}
Using H\"older's inequality, we have
\begin{equation*}
C_{\mathrm{Bre}}(\mathcal H^n_g(\Omega))^{\frac{n-1}n}\leq \mathcal H^{n-1}_g(\partial\Omega)+\|H\|_{L^n(\Sigma)}(\mathcal H^n_g(\Omega))^{\frac{n-1}n}.
\end{equation*}
Therefore,
\begin{equation*}
\left(C_{\mathrm{Bre}}-\|H\|_{L^n(\Sigma)}\right)(\mathcal H^n_g(\Omega))^{\frac{n-1}n}\leq \mathcal H^{n-1}_g(\partial\Omega).
\end{equation*}
This completes the proof of the inequality.

If $\Omega$ is a flat, round $n$-ball and $M \simeq \mathbb{R}^{n+m}$, then equality follows immediately. For the characterization of the equality when $m=1$ or $m=2$, we notice that the case $m=1$ can be reduced to $m=2$ by considering the new set $\widetilde\Omega=\Omega\times\{0\}\subset M\times\mathbb R=\widetilde M$. The conclusion in codimension two follows from \cite[Theorem 1.6]{MR4612577}. This concludes the proof of the proposition.
\end{proof}

Since $\Sigma$ may satisfy $\partial\Sigma=\emptyset$, the reader might worry that closed examples (such as the $n$-dimensional sphere $\mathbb S^n\subset\mathbb R^{n+1}$) provide a counterexample to Proposition~\ref{propIP} with $\Omega=\Sigma$. This is not the case: the assumption $\|H\|_{L^n(\Sigma)}<C_{\mathrm{Bre}}$ can never hold when $\Sigma$ is closed. For example, for $\mathbb S^{n}$ we have
\begin{equation*}
\dfrac{\|H\|_{L^n(\Sigma)}}{C_{\mathrm{Bre}}}=\left(\dfrac{\omega_n}{|B^n|}\right)^{\frac1n}>1.
\end{equation*}
Indeed, applying $u\equiv 1$ on Theorem \ref{theobrendle} guarantees the following corollary.
\begin{cor}\label{cor319}
Let $M$ be a complete noncompact manifold of dimension $n+m$ with nonnegative sectional curvature. Assume that there exists $\Sigma\subset M$ a closed $n$-dimensional submanifold. Then
\begin{equation*}
\int_\Sigma|H|\mathrm d\mathcal H^n_g\geq C_{\mathrm{Bre}}(\mathcal H^n_g(\Sigma))^{\frac{n-1}{n}}.
\end{equation*}
In particular, $\|H\|_{L^n(\Sigma)}\geq C_{\mathrm{Bre}}$.
\end{cor}

Here we consider the Schwarz symmetrization used in \cite{MR4977133,MR4566705} which was inspired by the classical works \cite{MR448404,MR1704184}. For a more modern treatment, we refer the reader to the book by Kesavan \cite[Chapter 1]{MR2238193} and to the book by Lieb and Loss \cite[Chapter 3]{MR1817225}.

Let $(M,g)$ be an $(n+m)$-dimensional Riemannian manifold, and let $\Sigma\subset M$ be an $n$-dimensional submanifold. Let $u\colon\Sigma\to\mathbb R$ be a measurable function. For $t\in\mathbb R$, the superlevel set $\{u>t\}$ is defined by
\begin{equation*}
\{u>t\}=\{x\in\Sigma\colon u(x)>t\}.
\end{equation*}
The sets $\{u<t\}$, $\{u\geq t\}$, $\{u=t\}$, and similar variants are defined analogously. We say that a measurable function $u\colon\Sigma\to\mathbb R$ is \textbf{fast decaying} if $\mathcal H^n_g(\{|u|>t\})<\infty$ for every $t>0$. For such a function, we associate its \textbf{Schwarz rearrangement} $u^*\colon\mathbb R^n\to[0,\infty)$ by
\begin{equation*}
u^*(x)=u^{\#}(|B^n||x|^n),
\end{equation*}
where $u^{\#}\colon[0,\mathcal H^n_g(\Sigma)]\to[0,\infty]$ is the \textbf{one-dimensional decreasing rearrangement} of $|u|$, defined by
\begin{equation*}
\left\{\begin{array}{l}
u^{\#}(0)=\mathrm{ess\,sup}(|u|)\\
u^{\#}(s)=\inf\{t\geq0\colon\mathcal H^n_g(\{|u|>t\})<s\},\quad\forall s>0.
\end{array}\right.
\end{equation*}
The explicit formula of $u^*$ will not be needed here. For readers interested in the full set of properties that follows from this definition, we again refer to \cite[Chapter 1]{MR2238193}. The only relevant results for our purposes are that $u^*\colon\mathbb R^n\to[0,\infty]$ is radially symmetric and nonincreasing. Moreover, $|u|$ and $u^*$ are equimeasurable with respect $\mathcal H^n_g$ and $\mathcal L^n$, namely,
\begin{equation}\label{equimeasurable}
\mathcal H^n_g(\{|u|>t\})=\mathcal L^n(\{u^*>t\}),\quad\forall t\geq0,
\end{equation}
where $\mathcal L^n$ denotes the Lebesgue measure on $\mathbb R^n$.

Using \eqref{equimeasurable} together with the layer cake representation, we obtain the following proposition.
\begin{prop}\label{propeqm}
Let $(M,g)$ be an $(n+m)$-dimensional Riemannian manifold, and let $\Sigma\subset M$ be an $n$-dimensional submanifold. Let $u\colon\Sigma\to\mathbb R$ be a fast decaying function. Let $\Phi\colon[0,\infty)\to\mathbb R$ be a Borel measurable function with $\Phi(0)=0$. If $\Phi\geq0$ or $\Phi(|u|)\in L^1(\Sigma)$, then
\begin{equation*}
\int_\Sigma\Phi(|u(x)|)\mathrm d\mathcal H^n_g=\int_{\mathbb R^n}\Phi(u^*(x))\mathrm d\mathcal L^n.
\end{equation*}

In particular, given $p\in(0,\infty)$, we have
\begin{equation*}
\int_\Sigma|u|^p\mathrm d\mathcal H^n_g=\int_{\mathbb R^n}|u^*|^p\mathrm d\mathcal L^n.
\end{equation*}
\end{prop}

Fix $1 \leq p<\infty$. We say that a function $u \in L^p(\Sigma)$ is $p$-Sobolev and write $u \in W^{1, p}(\Sigma)$ if there exists a vector valued function $\nabla^{\Sigma} u \in L^p(\Sigma)$, called weak gradient of $u$, such that
\begin{equation*}
\int_{\Sigma} u \operatorname{div}_{\Sigma} X\mathrm d\mathcal H^n_g=-\int_{\Sigma} X \cdot \nabla^{\Sigma} u\,\mathrm d\mathcal H^n_g, \quad \forall X \in \mathfrak{X}_0(\Sigma),
\end{equation*}
where $\mathfrak{X}_0(\Sigma)$ is the set of tangent vector fields on $\Sigma$ that vanish on $\partial \Sigma$. We then define the $p$-Sobolev norm $\|\cdot\|_{W^{1, p}}$ in $W^{1, p}(\Sigma)$ to be the completion of $C^\infty_0(\Sigma)$ with the following norm
\begin{equation*}
\|u\|_{W^{1, p}(\Sigma)}:=\|u\|_{L^p(\Sigma)}+\|\nabla^{\Sigma} u\|_{L^p(\Sigma)}.
\end{equation*}
Finally, the family $W_0^{1, p}(\Sigma)$ of $p$-Sobolev functions vanishing at $\partial \Sigma$ is defined as the topological closure of $C_0^{\infty}(\Sigma) \subseteq W^{1, p}(\Sigma)$ with respect to the norm $\|\cdot\|_{W^{1,p}(\Sigma)}$.

\begin{proof}[Proof of Theorem \ref{MainPS}]
Since $\|\nabla^\Sigma|u|\|_{L^p(\Sigma)}=\|\nabla^\Sigma u\|_{L^p(\Sigma)}$, we may assume $u\geq0$. Moreover, by approximation, we may also assume $u\in C^\infty_0(\Sigma)$, and by restricting $\Sigma$ to the support of $u$ we may assume that $\Sigma$ is compact (possibly with boundary $\partial\Sigma$). Using \eqref{equimeasurable}, we have
\begin{equation}\label{equimeasurable2}
\mathcal H^n_g\left(\{u>t\}\right)=\mathcal H^n_\delta\left(\{u^*>t\}\right).
\end{equation}
Let $\mathrm{Crit}(u)$ and $\mathrm{Crit}(u^*)$ denote the sets of critical points of $u$ and $u^*$, respectively. By the coarea formula,
\begin{align}
\mathcal H^n_g(\{u>t\})&=\mathcal H^n_g\left(\{u>t\}\cap \mathrm{Crit}(u)\right)+\int_{\{u>t\}\backslash\mathrm{Crit}(u)}\dfrac{|\nabla^\Sigma u|}{|\nabla^\Sigma u|}\mathrm d\mathcal H^n_g\nonumber\\
&=\mathcal H^n_g\left(\{u>t\}\cap \mathrm{Crit}(u)\right)+\int_t^\infty\int_{\{u=s\}\backslash\mathrm{Crit}(u)}\dfrac{1}{|\nabla^\Sigma u|}\mathrm d\mathcal H^{n-1}_g\mathrm ds,\label{pse0}
\end{align}
for any $t\in\mathbb R$. We claim that
\begin{equation}\label{pse1}
-\dfrac{\mathrm d}{\mathrm dt}\mathcal H^n_g(\{u>t\})=\int_{\{u=t\}}\dfrac{1}{|\nabla^\Sigma u|}\mathrm d\mathcal H^{n-1}_g,\quad\mbox{for a.e. }t\in\mathbb R,
\end{equation}
and
\begin{equation}\label{pse2}
-\dfrac{\mathrm d}{\mathrm dt}\mathcal H^n_\delta(\{u^*>t\})=\dfrac{\mathcal H^{n-1}_\delta(\{u^*=t\})}{|\nabla u^*|_t},\quad\mbox{for a.e. }t\in\mathbb R,
\end{equation}
where $|\nabla u^*|_{t}$ denotes the value $|\nabla u^*(x)|$ for any $x\in\{u^*=t\}$. 

Let us first prove \eqref{pse1}. Since $u\in C^\infty_0(\Sigma)$, the set $\mathbb R\backslash u(\mathrm{Crit}(u))$ is open. Hence the derivative of $t\mapsto\mathcal H^n_g(\{u>t\}\cap\mathrm{Crit}(u))$ is zero for every $t\notin u(\mathrm{Crit}(u))$. Because $u$ is smooth, Sard's theorem implies that $u(\mathrm{Crit}(u))$ has zero measure. Applying the Lebesgue differentiation theorem to \eqref{pse0}, we get
\begin{equation*}
-\dfrac{\mathrm d}{\mathrm dt}\mathcal H^n_g(\{u>t\})=\int_{\{u=t\}\backslash\mathrm{Crit(u)}}\dfrac{1}{|\nabla^\Sigma u|}\mathrm d\mathcal H^{n-1}_g,\quad\mbox{for a.e. }t\in\mathbb R.
\end{equation*}
This completes the proof of \eqref{pse1}. 

Now we prove \eqref{pse2}. From \eqref{equimeasurable2} and $u^*$ being radial we know that $\{u^*>t\}=B_{\rho(t)}(0)$ with
\begin{equation*}
\rho(t)=\left(\dfrac{\mathcal H^n_g(\{u>t\})}{|B^n|}\right)^{\frac1n}.
\end{equation*}
By \eqref{pse1}, $\rho'(t)$ exists a.e. $r\in\mathbb R$ and
\begin{equation}\label{rhon}
\rho'(t)=\dfrac{1}{n|B^n|\rho^{n-1}(t)}\dfrac{\mathrm d}{\mathrm dt}\mathcal H^n_g(\{u>t\})<0,\quad\mbox{a.e. }t\in\mathbb R.
\end{equation}
Let $\phi$ with $u^*(x)=\phi(|x|)$. Note that
\begin{align*}
\phi(|x|)&=u^*(x)=u^{\#}(|B^n||x|^n)=\inf\{t\geq0\colon\mathcal H^n_g(\{u>t\})<|B^n||x|^n\}\\
&=\inf\{t\geq0\colon\rho(t)<|x|\}.
\end{align*}
Since $\rho$ is nonincreasing, given $t_0$ with $\rho'(t_0)<0$ we have $\rho$ is strictly decreasing on a neighborhood of $t_0$. By the equation above, $\phi(\rho(t_0))=t_0$. Using \eqref{rhon}, we obtain
\begin{equation*}
|\nabla u^*|_t=|\nabla u^*(x)|=\left|\phi'(|x|)\right|=\dfrac{1}{|\rho'(t)|}\quad\mbox{a.e. }t\in\mathbb R
\end{equation*}
where $x\in \mathbb R^n$ is any point with $u^*(x)=t$. Then, using that $\mathcal H_\delta^n(\{u^*>t\})=|B^n|\rho^n(t)$ and $\mathcal H^{n-1}_\delta(\{u^*=t\})=n|B^n|\rho^{n-1}(t)$,
\begin{equation*}
-\dfrac{\mathrm d}{\mathrm dt}\mathcal H^n_\delta(\{u^*>t\})=-n|B^n|\rho^{n-1}(t)\rho'(t)=\dfrac{\mathcal H^{n-1}_\delta(\{u^*=t\})}{|\nabla u^*|_t},\quad\mbox{a.e. }t\in\mathbb R.
\end{equation*}
This completes the proof of \eqref{pse2}.

Using the radial symmetry of $u^*$ and the fact that the balls attain equality in the Euclidean isoperimetric inequality, we have
\begin{equation*}
\mathcal H^{n-1}_\delta(\{u^*=t\})=n|B^n|^{\frac1n}\left(\mathcal H^n_\delta(\{u^*>t\})\right)^{\frac{n-1}n},\quad\forall t\notin u^*(\mathrm{Crit}(u^*)).
\end{equation*}
Here we used that $\{u^*=t\}$ is an $(n-1)$-dimensional manifold whenever $t$ is a regular value. By combining \eqref{equimeasurable2} with the isoperimetric inequality in Proposition \ref{propIP}, we obtain
\begin{equation}\label{1stEst}
\mathcal H^{n-1}_\delta(\{u^*=t\})=n|B^n|^{\frac1n}\left(\mathcal H^n_g(\{u>t\})\right)^{\frac{n-1}n}\leq n|B^n|^{\frac1n}C_{\mathrm{IP}}\mathcal H^{n-1}_g(\{u=t\}),
\end{equation}
for all $t\notin u(\mathrm{Crit}(u))\cup u^*(\mathrm{Crit}(u^*))$. For now, let us assume $p>1$. Applying H\"older's inequality and \eqref{pse1}, we get
\begin{align}
\mathcal H^{n-1}_\delta(\{u^*=t\})&\leq n|B^n|^{\frac1n}C_{\mathrm{IP}}\int_{\{u=t\}}\dfrac{1}{|\nabla^\Sigma u|^{\frac{p-1}p}}|\nabla^\Sigma u|^{\frac{p-1}p}\mathrm d\mathcal H^{n-1}_g \nonumber\\
&\leq n|B^n|^{\frac1n}C_{\mathrm{IP}}\left[-\dfrac{\mathrm d}{\mathrm dt}\mathcal H^n_g(\{u>t\})\right]^{\frac{p-1}p}\left(\int_{\{u=t\}}|\nabla^\Sigma u|^{p-1}\mathrm d\mathcal H^{n-1}_g\right)^{\frac1p}\label{2ndEst},
\end{align}
for a.e. $t\in\mathbb R$. Using \eqref{equimeasurable2} and \eqref{pse2}, we deduce
\begin{align*}
\int_{\{u=t\}}|\nabla^\Sigma u|^{p-1}\mathrm d\mathcal H^{n-1}_g&\geq\left[\left(n|B^n|^{\frac1n}C_{\mathrm{IP}}\right)^{-1}\mathcal H^{n-1}_\delta(\{u^*=t\})\right]^{p}\left[-\dfrac{\mathrm d}{\mathrm dt}\mathcal H^n_\delta(\{u^*>t\})\right]^{1-p}\\
&=\left(n|B^n|^{\frac1n}C_{\mathrm{IP}}\right)^{-p}|\nabla u^*|_t^{p-1}\mathcal H^{n-1}_\delta(\{u^*=t\})\\
&=\left(n|B^n|^{\frac1n}C_{\mathrm{IP}}\right)^{-p}\int_{\{u^*=t\}}|\nabla u^*|^{p-1}\mathrm d\mathcal H^{n-1}_\delta,
\end{align*}
for a.e. $t\in\mathbb R$. We now apply the coarea formula once more in the following way:
\begin{align*}
\int_\Sigma|\nabla^\Sigma u|^p\mathrm d\mathcal H^{n}_g&=\int_{-\infty}^\infty\int_{\{u=s\}}|\nabla^\Sigma u|^{p-1}\mathrm d\mathcal H^{n-1}_g\mathrm ds\\
&\geq\left(n|B^n|^{\frac1n}C_{\mathrm{IP}}\right)^{-p}\int_{-\infty}^\infty\int_{\{u^*=s\}}|\nabla u^*|^{p-1}\mathrm d\mathcal H^{n-1}_\delta\mathrm ds\\
&=\left(n|B^n|^{\frac1n}C_{\mathrm{IP}}\right)^{-p}\int_{\mathbb R^n}|\nabla u^*|^p\mathrm d\mathcal H^{n}_\delta.
\end{align*}
This concludes the case $p>1$. The case $p=1$ is easier and is a straightforward consequence of~\eqref{1stEst} and the coarea formula. Therefore, we finish the proof.
\end{proof}
\begin{remark}\label{remarkequality}
    We now analyze the equality case in the above P\'{o}lya–Szeg\"{o} inequality for $p>1$. Indeed, it occurs if and only if equality is attained simultaneously in estimates \eqref{1stEst} and \eqref{2ndEst}. For \eqref{1stEst}, Proposition \ref{propIP} implies that equality holds only if the superlevel set $\{u>t\}$ is a flat round $n$-ball and the ambient manifold $M$ is isometric to $\mathbb{R}^{n+m}$, $m\in\{1,2\}$. On the other hand, equality in \eqref{2ndEst} is achieved if and only if $|\nabla^\Sigma u|$ is constant on the level set $\{u=t\}$. Therefore, the equality holds precisely when $M$ is isometric to $\mathbb{R}^{n+m}$, $\Sigma \subset \mathbb{R}^{n+m}$ is a flat round $n$-ball, and $u$ is a radial function on $\Sigma$, that is, $$u(x)=h(d(x,o)),$$ 
    where $h:[0,\infty)\rightarrow \mathbb{R}$ is a $W^{1,p}$-function and $d$ denotes the geodesic distance on $\Sigma$ from a fixed pole $o$. In this situation, the optimal constant is given by 
    $$C_{\mathrm{PS}}=\dfrac{n|B^n|^{\frac1n}}{n\,\mathrm{AVR}_g^{\frac1n}|B^n|^{\frac1n}-\|H\|_{L^n(\Sigma)}}=1.$$

\end{remark}

\begin{proof}[Proof of Proposition \ref{propCPS}]
Since $C_{\mathrm{PS}}$ does not change for both $m=1$ and $m=2$ and the identity $(n+2)|B^{n+2}|=2|B^2||B^n|$ holds, we may assume $m\geq2$ and
\begin{equation*}
C_{\mathrm{PS}}=\dfrac{n|B^n|^{\frac1n}}{n\mathrm{AVR}_g^{\frac{1}{n}}\left(\frac{(n+m)|B^{n+m}|}{m|B^m|}\right)^{\frac1n}-\|H\|_{L^n(\Sigma)}}.
\end{equation*}
Note that
\begin{equation}\label{peq1}
C_{\mathrm{PS}}\geq\dfrac{n|B^n|^{\frac1n}}{n\mathrm{AVR}_g^{\frac{1}{n}}\left(\frac{(n+m)|B^{n+m}|}{m|B^m|}\right)^{\frac1n}}.
\end{equation}
Moreover, since $\mathrm{AVR}_g\leq1$, we also have
\begin{equation}\label{peq2}
\dfrac{n|B^n|^{\frac1n}}{n\mathrm{AVR}_g^{\frac{1}{n}}\left(\frac{(n+m)|B^{n+m}|}{m|B^m|}\right)^{\frac1n}}\geq\dfrac{n|B^n|^{\frac1n}}{n\left(\frac{(n+m)|B^{n+m}|}{m|B^m|}\right)^{\frac1n}}.
\end{equation}
Then,
\begin{equation*}
C_{\mathrm{PS}}\geq\left(\dfrac{m|B^n||B^m|}{(n+m)|B^{n+m}|}\right)^{\frac{1}{n}}.
\end{equation*}
Using the formula for the volume of a unit Euclidean ball,
\begin{equation*}
|B^n|=\dfrac{\pi^{\frac{n}{2}}}{\Gamma\left(\frac{n}{2}+1\right)},
\end{equation*}
where $\Gamma$ is the gamma function, we obtain
\begin{equation*}
C_{\mathrm{PS}}\geq\left(\dfrac{m\Gamma\left(\frac{n+m}2+1\right)}{(n+m)\Gamma\left(\frac{n}2+1\right)\Gamma\left(\frac{m}2+1\right)}\right)^{\frac1n}.
\end{equation*}
Equivalently,
\begin{equation*}
C_{\mathrm{PS}}\geq\left(\dfrac{2m\Gamma\left(\frac{n+m}2+2\right)}{(n+m)(n+m+2)\Gamma\left(\frac{n}2+1\right)\Gamma\left(\frac{m}2+1\right)}\right)^{\frac1n}.
\end{equation*}
Then, $C_{\mathrm{PS}}\geq1$ is a consequence of the following claim
\begin{equation}\label{peq3}
\beta\left(\frac{n}2+1,\frac{m}2+1\right)\leq\dfrac{2m}{(n+m)(n+m+2)},
\end{equation}
where $\beta$ is the beta function, defined by
\begin{equation*}
\beta(x,y)=\dfrac{\Gamma(x)\Gamma(y)}{\Gamma(x+y)}=\int_0^1t^{x-1}(1-t)^{y-1}\mathrm dt,\quad\forall x,y>0.
\end{equation*}

In order to prove the inequality \eqref{peq3}, we use the identities $\beta(x,y+1)=\frac{y}{x+y}\beta(x,y)$ and $\beta(x+1,y)=\frac{x}{x+y}\beta(x,y)$, which gives
\begin{align*}
\beta\left(\frac{n}2+1,\frac{m}2+1\right)&=\dfrac{nm}{(n+m)(n+m+2)}\beta\left(\frac{n}2,\frac{m}2\right)\\
&=\dfrac{nm}{(n+m)(n+m+2)}\int_0^1t^{\frac{n}{2}-1}(1-t)^{\frac{m}{2}-1}\mathrm dt.
\end{align*}
Since $m\geq2$, it follows that
\begin{equation}\label{peq4}
(1-t)^{\frac{m}{2}-1}\leq1,\quad\forall t\in(0,1).
\end{equation}
Hence,
\begin{equation*}
\beta\left(\frac{n}2+1,\frac{m}2+1\right)\leq \dfrac{nm}{(n+m)(n+m+2)}\int_0^1t^{\frac{n}{2}-1}\mathrm dt=\dfrac{2m}{(n+m)(n+m+2)}.
\end{equation*}
This completes the proof of \eqref{peq3}, and thus we have $C_{\mathrm{PS}}\geq1$.

Finally, we verify that $C_{\mathrm{PS}}=1$ if and only if $\mathrm{AVR}_g=1$, $m=1$ or $m=2$, and $\Sigma$ is minimal. Under these conditions, using the identity $(n+2)|B^{n+2}|=2|B^2||B^n|$, one directly checks that $C_{\mathrm{PS}}=1$. Conversely, if $C_{\mathrm{PS}}=1$, all the inequalities above must be equalities. In particular, \eqref{peq1} implies that $\Sigma$ is minimal, \eqref{peq2} guarantees that $\mathrm{AVR}_g=1$, and \eqref{peq4} obtains that $m=2$ (within the case $m\geq2$ treated here; the case $m=1$ is the remaining possibility and is checked directly from the first branch of \eqref{CPS}). This completes the proof.
\end{proof}

\section{Sobolev, log-Sobolev, Gagliardo--Nirenberg, and Hardy Inequalities}\label{sec3}

In this section, we show how the P\'olya--Szeg\"o inequality obtained above transfers several sharp Euclidean functional inequalities to the submanifold setting. The first three consequences (Sobolev, logarithmic Sobolev, and Gagliardo--Nirenberg inequalities) follow directly from the equimeasurability of Schwarz rearrangements (see Proposition~\ref{propeqm}) together with Theorem~\ref{MainPS}. The Hardy inequality requires one additional ingredient, since the singular weight depending on the intrinsic distance to a fixed point cannot be transported by equimeasurability alone. We begin with the Sobolev inequality. Recall that Talenti~\cite{MR0463908} proved the sharp Euclidean estimate:
$$\|v\|_{L^{p^*}(\mathbb{R}^n)} \leq TA(n,p) \|\nabla v\|_{L^p(\mathbb{R}^n)},$$
for all $v \in W^{1,p}(\mathbb{R}^n)$, with $1<p<n$ and the sharp constant given by
$$TA(n,p)=\dfrac{1}{\sqrt{\pi}n^{\frac1p}}\left(\dfrac{p-1}{n-p}\right)^{1-\frac{1}{p}}\left(\dfrac{\Gamma (1+\frac{n}{2})\Gamma(n)}{\Gamma(\frac{n}{p})\Gamma(1+n-\frac{n}{p})}\right)^{\frac{1}{n}}.$$
As a direct consequence of the classical Sobolev inequality and Theorem \ref{MainPS}, we obtain a Sobolev inequality on submanifolds:

\begin{proof}[Proof of Corollary \ref{corSob}]
    It follows directly from the P\'{o}lya--Szeg\"{o} inequality and Talenti's Sobolev embedding on the Euclidean space. More clearly,
    \begin{align*}
        \|u\|_{L^{p^*}(\Sigma)}=\|u^*\|_{L^{p^*}(\mathbb{R}^n)} \leq TA(n,p) \|\nabla u^*\|_{L^p(\mathbb{R}^n)}\leq TA(n,p)C_{\mathrm{PS}}\|\nabla^{\Sigma}u\|_{L^p(\Sigma)},
    \end{align*}
    as desired.
\end{proof}

Now, we prove the Corollary \ref{corLog} about the log-Sobolev inequality on submanifolds, which is obtained in the same way as the previous one.
\begin{proof}[Proof of Corollary \ref{corLog}]
By density, it is enough to prove the result for $u\in C^\infty_0(\Sigma)$ with $\|u\|_{L^p(\Sigma)}=1$. First, we have
\[
\int_{\Sigma} |u|^p\log |u|\mathrm d\mathcal H^n_g=\int_{\mathbb R^n} |u^*|^p\log |u^*|\mathrm dx,
\]
for all $p \geq 1$ (see \cite{do2026logarithmicsobolevpoincarebeckner} for details). Then, by applying the $L^p$ log-Sobolev inequality in the Euclidean space $\mathbb{R}^n$ to $u^*$ (see \cite{MR1957678}), we get
\begin{equation*}
\int_{\Sigma} |u|^p\log|u|\mathrm d\mathcal{H}^n_g\leq \dfrac{n}{p^2} \ln \left[\mathcal{L}_{n,p} \int_{\mathbb{R}^n}|\nabla u^*|^p\mathrm dx\right]\leq \dfrac{n}{p^2}\ln\left[\mathcal{L}_{n,p}C_{\mathrm{PS}}^p\int_{\Sigma}|\nabla^\Sigma u|^p\mathrm d\mathcal{H}^n_g\right],
\end{equation*}
Here, the last inequality follows from Theorem~\ref{MainPS}. This completes the proof.
\end{proof}

We now focus on proving Corollary \ref{cor13}. We begin by recalling the classical Gagliardo--Nirenberg inequality in the Euclidean space; see \cite[Theorems 1 and 2]{MR1940370}.
\begin{theoremletter}[Gagliardo--Nirenberg inequality on the Euclidean space]
Assume that $1<p<n$ and $1<q<\frac{p(n-1)}{n-p}$. Then the following statements hold.\label{GNE}

\smallskip

\noindent (i) If $q>p$, then for every $u\in C^\infty_0(\mathbb R^n)$,
\begin{equation*}
\|u\|_{L^{p\frac{q-1}{p-1}}(\mathbb R^n)}\leq \mathrm{GN}_1(n,p,q)\|\nabla u\|^\theta_{L^p(\mathbb R^n)}\|u\|_{L^q(\mathbb R^n)}^{1-\theta},
\end{equation*}
where $\theta=\frac{(q-p)n}{(q-1)(np+pq-nq)}$ and the sharp constant $\mathrm{GN}_1=\mathrm{GN}_1(n,p,q)$ is
\begin{equation*}
\mathrm{GN}_1\!=\!\left(\dfrac{q-p}{p\sqrt\pi}\right)^\theta\!\left(\dfrac{pq}{n(q-p)}\right)^{\frac{\theta}{p}}\!\left(\dfrac{np+pq-nq}{pq}\right)^{\frac{p-1}{p(q-1)}}\!\left(\dfrac{\Gamma\left(q\frac{p-1}{q-p}\right)\Gamma\left(\frac{n}{2}+1\right)}{\Gamma\left(\frac{p-1}{p}\frac{np+pq-nq}{q-p}\right)\!\Gamma\left(n\frac{p-1}{p}+1\right)}\!\right)^{\frac{\theta}{n}}\!.
\end{equation*}

\smallskip

\noindent (ii) If $q<p$, then for every $u\in C^\infty_0(\mathbb R^n)$,
\begin{equation*}
\|u\|_{L^{q}(\mathbb R^n)}\leq \mathrm{GN}_2(n,p,q)\|\nabla u\|^\theta_{L^p(\mathbb R^n)}\|u\|_{L^{p\frac{q-1}{p-1}}(\mathbb R^n)}^{1-\theta},
\end{equation*}
where $\theta=\frac{(p-q)n}{q(np+pq-nq-p)}$ and the sharp constant $\mathrm{GN}_2=\mathrm{GN}_2(n,p,q)$ is
\begin{equation*}
\mathrm{GN}_2\!=\!\left(\!\dfrac{p-q}{p\sqrt\pi}\!\right)^\theta\!\left(\!\dfrac{pq}{n(p-q)}\!\right)^{\frac{\theta}{p}}\!\left(\!\dfrac{pq}{np+pq-nq}\!\right)^{\frac{1-\theta}{p}\frac{p-1}{q-1}}\!\left(\!\dfrac{\Gamma\left(\frac{p-1}{p}\frac{np+pq-nq}{p-q}+1\right)\Gamma\left(\frac{n}{2}+1\right)}{\Gamma\left(q\frac{p-1}{p-q}+1\right)\Gamma\left(n\frac{p-1}{p}+1\right)}\!\right)^{\frac{\theta}{n}}\!.
\end{equation*}
\end{theoremletter}

The preceding Euclidean inequalities now transfer directly to $\Sigma$. The only change in the constants comes from the P\'olya--Szeg\"o constant, and it appears with exponent $\theta$, since the gradient norm enters the Gagliardo--Nirenberg inequalities with this exponent.

\begin{proof}[Proof of Corollary \ref{cor13}]
Assume first that the conditions in \textit{(i)} hold. By Proposition~\ref{propeqm}, Theorem~\ref{GNE}, and Theorem~\ref{MainPS}, we obtain
\begin{equation*}
\|u\|_{L^{p\frac{q-1}{p-1}}(\Sigma)}=\|u^*\|_{L^{p\frac{q-1}{p-1}}(\mathbb R^n)}\leq\mathrm{GN}_1\|\nabla u^*\|_{L^p(\mathbb R^n)}^\theta\|u^*\|^{1-\theta}_{L^q(\mathbb R^n)}\leq\mathrm{GN}_1C_{\mathrm{PS}}^\theta\|\nabla^\Sigma u\|^\theta_{L^p(\Sigma)}\|u\|^{1-\theta}_{L^q(\Sigma)}.
\end{equation*}

The proof of \textit{(ii)} follows in the same way. Indeed, using the same results, we have
\begin{equation*}
\|u\|_{L^{q}(\Sigma)}=\|u^*\|_{L^{q}(\mathbb R^n)}\leq\mathrm{GN}_2\|\nabla u^*\|_{L^p(\mathbb R^n)}^\theta\|u^*\|^{1-\theta}_{L^{p\frac{q-1}{p-1}}(\mathbb R^n)}\leq\mathrm{GN}_2C_{\mathrm{PS}}^\theta\|\nabla^\Sigma u\|^\theta_{L^p(\Sigma)}\|u\|^{1-\theta}_{L^{p\frac{q-1}{p-1}}(\Sigma)}.
\end{equation*}
This proves the corollary.
\end{proof}

We now turn to the Hardy inequality. This part is slightly different from the previous applications. The Sobolev, logarithmic Sobolev, and Gagliardo--Nirenberg inequalities involve only norms of $u$ and $\nabla^\Sigma u$, and hence they follow from Proposition~\ref{propeqm} and Theorem~\ref{MainPS}. On the other hand, the Hardy inequality contains the weighted term
\begin{equation*}
\int_\Sigma \frac{|u|^p}{d(x,o)^p}\mathrm d\mathcal H_g^n. 
\end{equation*}
Thus, to compare it with its Euclidean counterpart, we also need a rearrangement estimate for products and a comparison for the rearrangement of the distance weight.

In \cite{MR4435963}, Cabr\'{e} and Miraglio established Hardy and Hardy--Sobolev inequalities on hypersurfaces of $\mathbb{R}^{n+1}$. Both inequalities involve a mean curvature term that captures the geometric features of the hypersurface. We note that the P\'{o}lya–Szeg\"{o} inequality we have established applies to submanifolds of arbitrary codimension. Consequently, under an appropriate bound on the mean curvature, we are able to derive Hardy and Hardy–Littlewood inequalities on more general submanifolds, not restricted to codimension one. The assumption $\mathrm{Ric}_\Sigma\geq0$ will be used to compare the intrinsic volume growth of geodesic balls in $\Sigma$ with the Euclidean one. We first recall the Hardy--Littlewood rearrangement inequality in our setting.

\begin{lemma}[Hardy--Littlewood inequality]\label{lemmaHL}
    Suppose that $f$ and $g$ are two non-negative functions. Then,
    $$\int_{\Sigma} f(x)g(x)\mathrm d\mathcal{H}^n_g \leq \int_{\mathbb{R}^n}f^*(x)g^*(x)\mathrm d\mathcal{L}^n.$$
\end{lemma}
\begin{proof}
    From the layer-cake representation, we have
    \begin{align*}
        \int_{\Sigma} f(x)g(x)\mathrm d\mathcal{H}^n_g&=\int_\Sigma \int_0^\infty \int_0^\infty \chi_{\{f(x)>s\}}\chi_{\{g(x)>t\}}\mathrm ds\mathrm dt\mathrm d\mathcal{H}^n_g\\
        &=\int_0^\infty \int_0^\infty \mathcal H^n_g\left(\{f(x)>s\} \cap \{g(x)>t\}\right)\mathrm ds\mathrm dt\\
        &\leq \int_0^\infty \int_0^\infty\mathcal L^n\left(\{f^*(x)>s\} \cap \{g^*(x)>t\}\right)\mathrm ds\mathrm dt\\
        &=\int_{\mathbb{R}^n} \int_0^\infty \int_0^\infty \chi_{\{f^*(x)>s\}}\chi_{\{g^*(x)>t\}}\mathrm ds\mathrm dt\mathrm d\mathcal{L}^n\\
        &=\int_{\mathbb{R}^n}f^*(x)g^*(x)\mathrm d\mathcal{L}^n,
    \end{align*}
    as desired.
\end{proof}
\begin{remark}
    In the above lemma, when $\Sigma$ is bounded, the rearrangements $f^*$ and $g^*$ are originally defined on the Euclidean ball $B\subset\mathbb{R}^n$ satisfying $|B|=|\Sigma|$. By extending $f^*$ and $g^*$ by zero outside $B$, they may be viewed as equimeasurable functions on the whole space $\mathbb{R}^n$. Consequently, the Hardy-Littlewood inequality may be expressed equivalently as an integral over $\mathbb{R}^n$, as stated above.
\end{remark}

The second lemma we need uses the fact that the Schwarz rearrangement commutes with nondecreasing functions.
\begin{lemma}
    Let $u:\Sigma\to[0,\infty]$ be measurable and fast decaying, and let $F:[0,\infty)\to[0,\infty)$ be nondecreasing,left-continuous, and satisfy $F(0)=0$. Then
    $$F(u^*)=(F(u))^*.$$
\end{lemma}
\begin{proof}
    Since $F$ is nondecreasing, 
    \begin{align*}
        \mathcal{H}^n_g(\{F(u)>s\})=\mathcal{H}^n_g(\{u>F^{-1}(s)\})=\mathcal{L}^n(\{u^*>F^{-1}(s)\})=\mathcal{L}^n(\{F(u^*)>s\}),
    \end{align*}
    for all $s \geq 0$. Then,
    $$\mathcal{L}^n(\{F(u^*)>s\})=\mathcal{H}^n_g(\{F(u)>s\})=\mathcal{L}^n(\{(F(u))^*>s\}),$$
    which implies that $(F(u^*))^*=(F(u))^{**}.$ Hence, $F(u^*)=(F(u))^*$, as desired.
\end{proof} 

We can now prove the Hardy inequality on submanifolds; see Theorem \ref{theohardy}. The strategy is to apply the Euclidean Hardy inequality to $u^*$ and then use the two previous rearrangement lemmas, together with the Bishop--Gromov comparison theorem, to dominate the weighted integral on $\Sigma$ by the corresponding Euclidean weighted integral.
\begin{proof}[Proof of Theorem \ref{theohardy}]
    Using the P\'{o}lya--Szeg\"{o} inequality and the $L^p$-Hardy inequality on $\mathbb{R}^n$, we have
    $$C_{\mathrm{PS}}^p\|\nabla^\Sigma u\|_{L^p(\Sigma)}^p \geq \|\nabla u^*\|_{L^p(\mathbb{R}^n)}^p \geq \left|\dfrac{n-p}{p}\right|^p\int_{\mathbb{R}^n} \dfrac{|u^*|^p}{|x|^p}\mathrm d\mathcal{L}^n.$$
    By the Hardy-Littlewood inequality,
    $$\int_{\Sigma}\dfrac{|u|^p}{(d(x,o))^p}d\mathcal{H}_g^n \leq \int_{\mathbb{R}^n}(|u|^p)^* \left(\dfrac{1}{(d(x,o))^p}\right)^*\mathrm d\mathcal{L}^n.$$
    Define $h(x):=\dfrac{1}{(d(x,o))^p}$. Next, we compute $h^*(y)$. Indeed,
    \begin{align*}
        h^*(y)&=h^*(|y|)=h^{\#}(|B^n||y|^n)=\inf\{t: \mathcal{H}^n_g(h(x)>t)<|B^n||y|^n\}\\
        &=\inf \left\{t: \mathcal{H}^n_g \left(d(x,o) < \frac{1}{t^{1/p}}\right)<|B^n||y|^n\right\},
    \end{align*}
    and 
    \begin{equation*}
        \mathcal{H}^n_g \left(d(x,o) < \frac{1}{t^{1/p}}\right)=\int_0^{\frac{1}{t^{1/p}}}\mathcal{H}^{n-1}_g\left(\partial B^\Sigma_r(o)\right)\mathrm dr\leq \int_0^{\frac{1}{t^{1/p}}}\omega_{n-1}r^{n-1}\mathrm dr=|B^n|\dfrac{1}{t^{n/p}},
    \end{equation*}
    where we used $\mathrm{Ric}_\Sigma \geq 0$ and $\Sigma$ being complete to apply the Bishop-Gromov theorem. Then, for all $t>\frac{1}{|y|^p}$, $\mathcal{H}^n_g \left(d(x,o) < \frac{1}{t^{1/p}}\right)<|B^n||y|^n$, which implies that $h^*(y) \leq \frac{1}{|y|^p}$. Hence,
    $$\int_{\Sigma}\dfrac{|u|^p}{(d(x,o))^p}\mathrm d\mathcal{H}_g^n \leq \int_{\mathbb{R}^n}\dfrac{|u^*|^p}{|x|^p}\mathrm d\mathcal{L}^n,$$
    which allows us
    \begin{align*}
        C_{\mathrm{PS}}^p\|\nabla^\Sigma u\|_{L^p(\Sigma)}^p \geq \left|\dfrac{n-p}{p}\right|^p\int_{\mathbb{R}^n} \dfrac{|u^*|^p}{|x|^p}\mathrm d\mathcal{L}^n&\geq \left|\dfrac{n-p}{p}\right|^p \int_{\Sigma}\dfrac{|u|^p}{(d(x,o))^p}\mathrm d\mathcal{H}_g^n,
    \end{align*}
    as desired.
\end{proof}
Finally, by a standard interpolation argument, a Hardy--Sobolev inequality can be derived as follows.
\begin{cor}
    Let $M$ be a complete noncompact manifold of dimension $n+m$ with nonnegative sectional curvature. Let $\Sigma$ be a  complete submanifold of $M$ of dimension $n$ without boundary satisfying $\|H\|_{L^n(\Sigma)}<C_{\mathrm{Bre}}$ and $\text{Ric}_{\Sigma}\geq 0$. Then, for all $u \in W^{1,p}_0(\Sigma)$ with $1<p<n$ and $\theta \in [0,1]$, 
$$\left(\int_\Sigma \dfrac{|u|^{\frac{p(n-\theta p)}{n-p}}}{(d(x,o))^{\theta p}}\mathrm d\mathcal{H}_g^n\right)^{\frac{n-p}{n-\theta p}}\leq C(n,p,\theta,C_{\mathrm{PS}})\|\nabla^\Sigma u\|_{L^p(\Sigma)}^p,$$
where $o$ is a fixed point in $\Sigma$.
\end{cor}
\begin{proof}
    For $\theta \in[0,1]$, by applying Holder's inequality, we have
    \begin{align*}
        \int_\Sigma \dfrac{|u|^{\frac{p(n-\theta p)}{n-p}}}{(d(x,o))^{\theta p}}\mathrm d\mathcal{H}_g^n &= \int_\Sigma \left(\dfrac{|u|}{d(x,o)}\right)^{\theta p}|u|^{(1-\theta)\frac{np}{n-p}}\mathrm d\mathcal{H}_g^n\\
        &\leq \left(\int_\Sigma \dfrac{|u|^p}{(d(x,o))^p}\mathrm d\mathcal{H}_g^n\right)^\theta \left(\int_\Sigma u^{p^*}\mathrm d\mathcal{H}_g^n\right)^{1-\theta}\leq C\|\nabla^\Sigma u\|_{L^p(\Sigma)}^{\theta p+(1-\theta)p^*}\\
        &= C\left(\|\nabla^\Sigma u\|_{L^p(\Sigma)}^{p}\right)^{\frac{n-\theta p}{n-p}},
    \end{align*}
    which implies our claim, since $\frac{n-p}{n-\theta p} \leq 1$.
\end{proof}
\begin{remark}\label{remarkhardy}
    In fact, using a standard rearrangement argument, we can get an explicit form of the constant $C(n,p)$ in the above corollary. Indeed, we have 
    \begin{align*}
        \int_{\Sigma} \dfrac{|u|^{\frac{p(n-\theta p)}{n-p}}}{(d(x,o))^{\theta p}}\mathrm d\mathcal{H}^n_g &\leq \int_{\mathbb{R}^n} \left(|u|^{\frac{p(n-\theta p)}{n-p}}\right)^*\left(\dfrac{1}{(d(x,o))^{\theta p}}\right)^*\mathrm d\mathcal{L}^n\\
        & \leq \int_{\mathbb{R}^n} |u^*|^{\frac{p(n-\theta p)}{n-p}}\dfrac{1}{|x|^{p\theta}}\mathrm d\mathcal{L}^n.
    \end{align*}
    Then,
    \begin{align*}
        \left(\int_{\Sigma} \dfrac{|u|^{\frac{p(n-\theta p)}{n-p}}}{(d(x,o))^{\theta p}}\mathrm d\mathcal{H}^n_g\right)^{\frac{n-p}{n-\theta p}}&\leq \left(\int_{\mathbb{R}^n} |u^*|^{\frac{p(n-\theta p)}{n-p}}\dfrac{1}{|x|^{p\theta}}\mathrm d\mathcal{L}^n\right)^{\frac{n-p}{n-\theta p}}\\
        &\leq \dfrac{1}{\mu(\mathbb{R}^n)}\int_{\mathbb{R}^n}|\nabla u^*|^p \mathrm d\mathcal{L}^n\\
        &\leq \dfrac{C_{\mathrm{PS}}^p}{\mu(\mathbb{R}^n)}\int_{\Sigma}|\nabla^{\Sigma} u|^p \mathrm d\mathcal{H}^n_g,
    \end{align*}
    which allows us to define $\mu(\mathbb{R}^n)$ as follows from a density argument:
    $$\mu(\mathbb{R}^n)=\inf \left\{\dfrac{\int_{\mathbb{R}^n}|\nabla u|^p\mathrm d\mathcal{L}^n}{\left(\int_{\mathbb{R}^n} |u|^{\frac{p(n-\theta p)}{n-p}} \frac{1}{|x|^{p\theta}}\mathrm d\mathcal{L}^n\right)^{\frac{n-p}{n-p\theta}}}: u \in W^{1,p}_0(\mathbb{R}^n) \setminus\{0\} \right\}.$$
    In particular, in the equality case of the P\'{o}lya--Szeg\"{o} inequality compatible with the hypotheses of the corollary, namely $\Sigma=\mathbb{R}^n\subset\mathbb{R}^{n+2}$ (the flat round $n$-ball of infinite radius, which is complete and without boundary), the constant $\frac{C_{\mathrm{PS}}^p}{\mu(\mathbb{R}^n)}$ becomes $\frac{1}{\mu(\mathbb{R}^n)}$, which is sharp and identical to the result given by Ghoussoub-Moradifam \cite[Theorem 15.2.2]{MR3052352}. Also, from \cite{MR3052352}, such a constant is only attained if $\Sigma =\mathbb{R}^n$. Moreover, in the general setting where $\Sigma$ is a submanifold satisfying $\|H\|_{L^n(\Sigma)} < C_{\mathrm{Bre}}$ and $\mathrm{Ric}_{\Sigma} \geq 0$, the obtained results refine those of Cabr\'{e} and Miraglio \cite{MR4435963}, although the sharpness of the constant is still unknown.
\end{remark}

\section{Moser--Trudinger Inequality on Submanifolds}\label{sec4}

In this section, our goal is to study when
$$\sup_{u\in\mathcal A} \int_\Sigma e^{\alpha|u|^{\frac{n}{n-1}}}<\infty,$$
for $\mathcal A=\{u \in C^\infty_0(\Sigma): \|\nabla^{\Sigma} u\|_{L^n(\Sigma)}\leq 1\}$. Precisely, we will present the proofs of Theorems~\ref{theoMoser} and~\ref{theoIM}. For Theorem~\ref{theoMoser}, corresponding to the Moser--Trudinger inequality, we again rely on the P\'{o}lya--Szeg\"{o} inequality established in order to reduce the problem to the Euclidean space, where the classical inequality is available. This allows us to transfer the Euclidean estimate to the submanifold $\Sigma$. On the other hand, for the supercritical case in Theorem~\ref{theoIM}, the argument is based on the classical Moser sequence in $\mathbb R^n$ transported to the submanifold $\Sigma$. Although the geometry of $\Sigma$ causes a distortion in this procedure, we can control it by slightly decreasing the exponent $\alpha>n\omega_{n-1}^{\frac1{n-1}}$, while still remaining in the supercritical regime. This establishes the sharpness of the Moser--Trudinger inequality on $\Sigma$ when $C_{\mathrm{PS}}=1$.

\begin{proof}[Proof of Theorem \ref{theoMoser}]
    Using Proposition \ref{propeqm} with $\Phi(t)=e^{\alpha t^{\frac{n}{n-1}}}-1$, and the P\'{o}lya--Szeg\"{o} inequality, we have
    \begin{align*}
        \sup_{u\in \mathcal{A}}\int_{\Sigma}e^{\alpha |u|^{\frac{n}{n-1}}}\mathrm d\mathcal H^n_g&=\sup_{u\in \mathcal{A}}\int_{\Sigma}\left(e^{\alpha |u|^{\frac{n}{n-1}}}-1\right)\mathrm d\mathcal H^n_g+\mathcal H^n_g(\Sigma)\\
        &=\sup_{u\in \mathcal{A}}\int_{D} e^{\alpha |u^*|^{\frac{n}{n-1}}}\mathrm d\mathcal{L}^n\leq \sup_{w \in \mathcal{B}}\int_{D}e^{\alpha |w|^{\frac{n}{n-1}}}\mathrm d\mathcal{L}^n\\
        &=\sup_{v \in \mathcal{C}}\int_{D}e^{\alpha C_{\mathrm{PS}}^{\frac{n}{n-1}} |v|^{\frac{n}{n-1}}}\mathrm d\mathcal{L}^n,
    \end{align*}
    where $D \subset \mathbb{R}^n$ is the ball centered at the origin having the volume $\mathcal H_g^n(\Sigma)$,
    $$\mathcal{B}:=\{w \in W^{1,n}_0(D):\|\nabla w\|_{L^n(D)} \leq C_{\mathrm{PS}}\}\;\;\text{and}\;\; \mathcal{C}:=\{v \in W^{1,n}_0(D):\|\nabla v\|_{L^n(D)} \leq 1\}.$$ 
    From the original Moser's result, for some $C_0>0$,
    \begin{align*}
        \sup_{u\in \mathcal{A}}\int_{\Sigma}e^{\alpha |u|^{\frac{n}{n-1}}}\mathrm d\mathcal H^n_g\leq \sup_{v \in \mathcal{C}}\int_{D}e^{\alpha C_{\mathrm{PS}}^{\frac{n}{n-1}} |v|^{\frac{n}{n-1}}}\mathrm d\mathcal{L}^n\leq C_0 \mathcal H_g^n(\Sigma),
    \end{align*}
    where $\alpha C_{\mathrm{PS}}^{\frac{n}{n-1}} \leq n\omega_{n-1}^{\frac{1}{n-1}}$, as desired.
\end{proof}

\begin{proof}[Proof of Theorem \ref{theoIM}]
From the sharpness of the Euclidean Moser inequality, for each $\ell\in\mathbb N$ there exists a radial function $\varphi_\ell\colon\mathbb R^n\to\mathbb R$ such that $\mathrm{supp}(\varphi_\ell)\subset B_{\frac1\ell}(0)$, $\|\nabla\varphi_\ell\|_{L^n(\mathbb R^n)}\leq1$, and
\begin{equation}\label{eqstar}
\int_{B_{\frac1\ell}(0)}e^{\alpha|\varphi_\ell|^{\frac{n}{n-1}}}\mathrm dx\overset{\ell\to\infty}\longrightarrow\infty,\quad\forall\alpha>n\omega_{n-1}^{\frac{1}{n-1}}.
\end{equation}
To construct such a sequence, define
\begin{equation}\label{MoserRn}
\varphi_\ell(x):=\left(\dfrac{\ell}{n}\right)^{\frac{n-1}{n}}\omega_{n-1}^{-\frac1n}\eta\left(\dfrac{t}{\ell}\right),
\end{equation}
where $e^{-t}=\frac{|x|^n}{R^n}$, $\eta(s)=\min\{s-L,1\}\chi_{\{s\geq L\}}(s)$, and $L=\frac{n}\ell\log[R(\ell+1)]$. Fix $x_0\in \Sigma$. Let
\begin{equation*}
\exp_{x_0}\colon B_R(0)\subset T_{x_0}\Sigma\simeq\mathbb R^n\to B_R(x_0)\subset \Sigma
\end{equation*}
be the exponential map, where $R>0$ is chosen small enough such that $\exp_{x_0}$ is a diffeomorphism. In normal coordinates, we have:
\begin{equation}\label{nc}
g_{ij}(x)=\delta_{ij}+O(|x|^2),\,g^{ij}(x)=\delta_{ij}+O(|x|^2),\mbox{ and }\sqrt{\det g(x)}=1+O(|x|^2).
\end{equation}
Taking $\ell$ large enough, we may assume $B_{\frac1\ell}(0)\subset B_R(0)$ to define $\psi_\ell\colon \Sigma\to\mathbb R$ by
\begin{equation}\label{MoserSigma}
\psi_\ell(y):=\left\{\begin{array}{ll}
     \varphi_\ell(\exp_{x_0}^{-1}(y)),&\mbox{if }y\in B_{\frac1\ell}(x_0),\\
     0,&\mbox{if }y\notin B_{\frac1\ell}(x_0). 
\end{array}\right.
\end{equation}
By \eqref{nc} and writing $x=\exp_{x_0}^{-1}(y)$, we obtain
\begin{equation*}
|\nabla\psi_\ell(y)|^2=\sum_{i,j=1}^ng^{ij}(x)\dfrac{\partial\varphi_\ell}{\partial x_i}(x)\dfrac{\partial\varphi_\ell}{\partial x_j}(x)\leq(1+C|x|^2)|\nabla\varphi_\ell(x)|^2.
\end{equation*}
Integrating,
\begin{equation}\label{psilleq}
\int_{\Sigma}|\nabla\psi_\ell(y)|^n\mathrm d\mathcal H^{n}_g(y)\leq\int_{B_{\frac1\ell}(0)}\left(1+C|x|^2\right)^{\frac{n}{2}}|\nabla\varphi_\ell(x)|^n\sqrt{\det g(x)}\mathrm dx.
\end{equation}
Using \eqref{nc} and $\|\nabla\varphi_\ell\|_{L^n(\mathbb R^n)}\leq1$, we have
\begin{equation*}
\int_{\Sigma}|\nabla\psi_\ell(y)|^n\mathrm d\mathcal H^{n}_g(y)\leq\left(1+\dfrac{C}{\ell^2}\right)^{\frac{n}{2}+1}\int_{B_{\frac1\ell}(0)}|\nabla\varphi_\ell(x)|^n\mathrm dx\leq\left(1+\frac{C}{\ell^2}\right)^{\frac{n+2}2}.
\end{equation*}
Replacing $\psi_\ell$ by $\widetilde\psi_\ell=\psi_\ell/(1+\frac{C}{\ell^2})^{\frac{n+2}{2n}}$, we get $\|\nabla\widetilde\psi_\ell\|_{L^n(\Sigma)}\leq1$. Moreover, since $\sqrt{\det g(x)}\geq1-C|x|^2$, we obtain
\begin{align}
\int_{\Sigma}e^{\alpha|\widetilde\psi_\ell|^{\frac{n}{n-1}}}\mathrm d\mathcal H^n_g&\geq\int_{B_{\frac1\ell}(0)}\exp\left(\dfrac{\alpha|\varphi_\ell|^{\frac{n}{n-1}}}{(1+\frac{C}{\ell^2})^{\frac{n+2}{2(n-1)}}}\right)\sqrt{\det g(x)}\mathrm dx\nonumber\\
&\geq\left(1-\dfrac{C}{\ell^2}\right)\int_{B_{\frac1\ell}(0)}\exp\left(\dfrac{\alpha|\varphi_\ell|^{\frac{n}{n-1}}}{(1+\frac{C}{\ell^2})^{\frac{n+2}{2(n-1)}}}\right)\mathrm dx.\label{eqasterix}
\end{align}
Given $\alpha>n\omega_{n-1}^{\frac1{n-1}}$, choose $\widetilde\alpha\in(n\omega_{n-1}^{\frac{1}{n-1}},\alpha)$ and $\ell$ sufficiently large such that
\begin{equation*}
1-\dfrac{C}{\ell^2}\geq\frac12\mbox{ and }\dfrac{\alpha}{\left(1+\frac{C}{\ell^2}\right)^{\frac{n+2}{2(n-1)}}}>\widetilde\alpha.
\end{equation*}
Then, by \eqref{eqstar} and \eqref{eqasterix},
\begin{equation*}
\int_{\Sigma}e^{\alpha|\widetilde\psi_\ell|^{\frac{n}{n-1}}}\mathrm d\mathcal H^n_g\geq\frac12\int_{B_{\frac1\ell}(0)}e^{\widetilde\alpha|\varphi_\ell|^{\frac{n}{n-1}}}\mathrm dx\overset{\ell\to\infty}\longrightarrow\infty.
\end{equation*}
This concludes the proof.
\end{proof}

\section{Exact Growth Inequality on Submanifolds}\label{sec5}

In this last section, we study the exact growth inequalities, where $\mathcal H^n_g(\Sigma)$ may be infinite. Specifically, we establish the exact growth inequality in our setting in the proof of Theorem~\ref{theoEG}, and then derive its consequences in Corollaries~\ref{cor1} and~\ref{cor2}. We also show that the inequality in Theorem~\ref{theoEG} fails whenever $\alpha>n\omega_{n-1}^{\frac1{n-1}}$ or $q<\frac{n}{n-1}$. As before, we do not obtain any result in the range~\eqref{rangeconj}.

\begin{proof}[Proof of Theorem \ref{theoEG}]
Let $u\in W^{1,n}_0(\Sigma)$ satisfying $\|\nabla^\Sigma u\|_{L^n(\Sigma)}\leq1$, and let $u^*$ denote its Schwarz rearrangement. By Propositions \ref{propeqm} and \ref{propCPS}, we obtain
\begin{align*}
\int_\Sigma\dfrac{\exp_n\left(\alpha_\Sigma|u|^{\frac{n}{n-1}}\right)}{(1+|u|)^{\frac{n}{n-1}}}\mathrm d\mathcal H^n_g&=\int_{\mathbb R^n}\dfrac{\exp_n\left(\alpha_\Sigma C_{\mathrm{PS}}^{\frac{n}{n-1}}\left|\frac{u^*}{C_{\mathrm{PS}}}\right|^{\frac{n}{n-1}}\right)}{(1+|u^*|)^{\frac{n}{n-1}}}\mathrm d\mathcal L^n\\
&\leq \int_{\mathbb R^n}\dfrac{\exp_n\left(\alpha_\Sigma C_{\mathrm{PS}}^{\frac{n}{n-1}}\left|\frac{u^*}{C_{\mathrm{PS}}}\right|^{\frac{n}{n-1}}\right)}{\left(1+\left|\frac{u^*}{C_{\mathrm{PS}}}\right|\right)^{\frac{n}{n-1}}}\mathrm d\mathcal L^n.
\end{align*}
Moreover, Theorem \ref{MainPS} ensures that $\|\nabla u^*\|_{L^n(\mathbb R^n)}\leq C_{\mathrm{PS}}$. Consequently, we may apply the Euclidean exact growth inequality due to Masmoudi and Sani \cite[Theorem 1.1]{MR3355498}, which yields
\begin{equation*}
\int_\Sigma\dfrac{\exp_n\left(\alpha_\Sigma|u|^{\frac{n}{n-1}}\right)}{(1+|u|)^{\frac{n}{n-1}}}\mathrm d\mathcal H^n_g\leq C\|u^*\|_{L^n(\mathbb R^n)}^n=C\|u\|_{L^n(\Sigma)}^n.
\end{equation*}
This completes the proof.
\end{proof}

We now proceed to prove the corollaries, which are related to the inequalities of Adachi--Tanaka \cite{MR1646323} and Li--Ruf \cite{MR2400264}.
\begin{proof}[Proof of Corollary \ref{cor1}]
Fix $\alpha\in[0,\alpha_\Sigma)$ and define the auxiliary function $\phi_\alpha\colon(0,\infty)\to\mathbb R$ given by
\begin{equation*}
\phi_\alpha(t)=\dfrac{(1+t)^{\frac{n}{n-1}}\exp_n(\alpha t^{\frac{n}{n-1}})}{\exp_n(\alpha_\Sigma t^{\frac{n}{n-1}})},\quad\forall t>0.
\end{equation*}
Since $\alpha<\alpha_\Sigma$, one can notice that
\begin{equation*}
\lim_{t\to0}\phi_\alpha(t)=\left(\dfrac{\alpha}{\alpha_\Sigma}\right)^{n-1}\mbox{ and }\lim_{t\to\infty}\phi_\alpha(t)=0.
\end{equation*}
As a consequence, there exists $C_\alpha>0$ such that
\begin{equation*}
\exp_n(\alpha t^{\frac{n}{n-1}})\leq C_\alpha\dfrac{\exp_n(\alpha_\Sigma t^{\frac{n}{n-1}})}{(1+t)^{\frac{n}{n-1}}},\quad\forall t\geq0.
\end{equation*}
The conclusion now follows directly from Theorem \ref{theoEG}. This concludes the proof.
\end{proof}

\begin{proof}[Proof of Corollary \ref{cor2}]
Let $u\in W^{1,n}_0(\Sigma)$ satisfying $\|\nabla^\Sigma u\|_{L^n(\Sigma)}^n+\tau\|u\|^n_{L^n(\Sigma)}\leq1$. We divide the proof into two cases.

\vspace{0.1cm}

\noindent\underline{Case 1:} $\|u\|^n_{L^n(\Sigma)}\geq\frac{n-1}{\tau n}$.

\vspace{0.1cm}

In this case, we have $\|\nabla^\Sigma(n^{\frac1n}u)\|^n_{L^n(\Sigma)}\leq n(1-\tau\|u\|^n_{L^n(\Sigma)})\leq1$. Then, we may apply Corollary \ref{cor1} to $n^{\frac1n}u$ to obtain $C_\alpha>0$ such that
\begin{equation*}
\int_\Sigma\exp_n\left(n^{\frac{1}{n-1}}\alpha|u|^{\frac{n}{n-1}}\right)\mathrm d\mathcal H^n_g\leq nC_\alpha\|u\|^n_{L^n(\Sigma)}\leq \dfrac{nC_\alpha}{\tau},\quad\forall \alpha\in[0,\alpha_\Sigma).
\end{equation*}
Choosing $\alpha=n^{-\frac1{n-1}}\alpha_\Sigma$ in the above inequality concludes this case.

\vspace{0.1cm}

\noindent\underline{Case 2:} $\|u\|^n_{L^n(\Sigma)}<\frac{n-1}{\tau n}$.

\vspace{0.1cm}

We first note that
\begin{align*}
\int_{\{|u|<1\}}\exp_n\left(\alpha_\Sigma|u|^{\frac{n}{n-1}}\right)\mathrm d\mathcal H_g^n&\leq\dfrac{\exp_n(\alpha_\Sigma)}{\alpha_\Sigma^{n-1}}\int_{\{|u|<1\}}\left(\alpha_\Sigma|u|^{\frac{n}{n-1}}\right)^{n-1}\mathrm d\mathcal H^n_g\\
&=\exp_n(\alpha_\Sigma)\int_{\{|u|<1\}}|u|^n\mathrm d\mathcal H^n_g\leq\exp_n(\alpha_\Sigma)\|u\|^n_{L^n(\Sigma)}\\
&\leq\dfrac{n-1}{\tau n}\exp_n(\alpha_\Sigma).
\end{align*}
Thus, it remains to prove that
\begin{equation}\label{eqpp2}
\int_{\{|u|\geq1\}}\exp_n\left(\alpha_\Sigma|u|^{\frac{n}{n-1}}\right)\mathrm d\mathcal H^n_g\leq C.
\end{equation}
We first claim that
\begin{equation}\label{eqpp3}
\left[\exp_n(t)\right]^p\leq\exp_n(pt),\quad\forall t\geq0,\ p\geq1.
\end{equation}
Indeed, a direct computation guarantees
\begin{equation*}
\dfrac{\mathrm d}{\mathrm dt}\left(\dfrac{\left[\exp_n(t)\right]^p}{\exp_n(pt)}\right)=\dfrac{p\left[\exp_n(t)\right]^{p-1}t^{n-2}}{\left[\exp_n(pt)\right]^2(n-2)!}\left[\exp_{n-1}(pt)-p^{n-2}\exp_{n-1}(t)\right]\geq0
\end{equation*}
and, by \cite[Lemma 3.1]{MR4819618},
\begin{equation*}
\dfrac{\left[\exp_n(t)\right]^p}{\exp_n(pt)}=\dfrac{\left[1-\frac{n-1}{(n-1)!}\int_t^\infty s^{n-2}e^{-s}\mathrm ds\right]^p}{1-\frac{n-1}{(n-1)!}\int_{pt}^\infty s^{n-2}e^{-s}\mathrm ds}\overset{t\to\infty}\longrightarrow1.
\end{equation*}
This proves \eqref{eqpp3}.

Using \eqref{eqpp3} and denoting $q=\frac{n-1}{n-1-\tau\|u\|^n_{L^n(\Sigma)}}>1$, we obtain
\begin{align}
\int_{\{|u|\geq1\}}&\exp_n\left(\alpha_\Sigma|u|^{\frac{n}{n-1}}\right)\mathrm d\mathcal H^n_g=\int_{\{|u|\geq1\}}\dfrac{\exp_n\left(\alpha_\Sigma|u|^{\frac{n}{n-1}}\right)}{(1+|u|)^{\frac{n}{n-1}\frac1q}}(1+|u|)^{\frac{n}{n-1}\frac1q}\mathrm d\mathcal H^n_g\nonumber\\
&\leq\left(\int_{\{|u|\geq1\}}\dfrac{\exp_n\left(q\alpha_\Sigma|u|^{\frac{n}{n-1}}\right)}{(1+|u|)^{\frac{n}{n-1}}}\mathrm d\mathcal H^n_g\right)^{\frac{1}{q}}\left(\int_{\{|u|\geq1\}}(1+|u|)^{\frac{n}{(n-1)(q-1)}}\mathrm d\mathcal H^n_g\right)^{\frac{q-1}q}\label{eqpp4}\\
&\leq2^{\frac{n}{(n-1)q}}\!\left(q\!\int_{\{|u|\geq1\}}\!\dfrac{\exp_n\!\left(\alpha_\Sigma|q^{\frac{n-1}{n}}u|^{\frac{n}{n-1}}\right)}{(1+|q^{\frac{n-1}n}u|)^{\frac{n}{n-1}}}\mathrm d\mathcal H^n_g\!\right)^{\frac{1}{q}}\!\left(\!\int_{\Sigma}|u|^{\frac{n}{(n-1)(q-1)}}\!\mathrm d\mathcal H^n_g\!\right)^{\frac{q-1}q}\!.\nonumber
\end{align}
On the other hand, using the inequality $(1-x)^s\leq1-sx$ for all $s,x\in[0,1]$, we have
\begin{align}
\|\nabla(q^{\frac{n-1}n}u)\|^n_{L^n(\Sigma)}&\leq q^{n-1}(1-\tau\|u\|^n_{L^n(\Sigma)})=\left(\dfrac{1}{1-\frac{\tau}{n-1}\|u\|^n_{L^n(\Sigma)}}\right)^{n-1}(1-\tau\|u\|^n_{L^n(\Sigma)})\nonumber\\
&=\left[\dfrac{(1-\tau\|u\|^n_{L^n(\Sigma)})^{\frac{1}{n-1}}}{1-\frac{\tau}{n-1}\|u\|^n_{L^n(\Sigma)}}\right]^{n-1}\leq\left[\dfrac{1-\frac{\tau}{n-1}\|u\|^n_{L^n(\Sigma)}}{1-\frac{\tau}{n-1}\|u\|^n_{L^n(\Sigma)}}\right]^{n-1}=1.\label{eqpp5}
\end{align}
Since $\|u\|^n_{L^n(\Sigma)}<\frac{n-1}{\tau n}$, we have $\frac{q}{q-1}>n$. Fix once and for all
\begin{equation}\label{alpha0def}
\alpha_0:=\min\left\{1,\dfrac{\alpha_\Sigma}2\right\},\qquad\mbox{so that }0<\alpha_0<\alpha_\Sigma\mbox{ and }\alpha_0\leq1.
\end{equation}
By \cite[Lemma 3.1]{MR4819618}, namely $t^{\gamma-1}\leq\Gamma(\gamma)\exp_\gamma(t)$, applied with $\gamma=\frac{q}{q-1}$ (so that $\gamma-1=\frac{1}{q-1}$) to $t=\alpha_0|u|^{\frac{n}{n-1}}$, we get
\begin{equation*}
\alpha_0^{\frac1{q-1}}|u|^{\frac{n}{(n-1)(q-1)}}\leq\Gamma\left(\dfrac{q}{q-1}\right)\exp_{\frac{q}{q-1}}\left(\alpha_0|u|^{\frac{n}{n-1}}\right).
\end{equation*}
Hence, using that $p\mapsto\exp_p(t)$ is nonincreasing together with $\frac{q}{q-1}>n$, and then Corollary \ref{cor1} with the exponent $\alpha_0<\alpha_\Sigma$ applied to $u$ itself (which is legitimate since $\|\nabla^\Sigma u\|_{L^n(\Sigma)}\leq1$), it follows that
\begin{align}
\int_\Sigma|u|^{\frac{n}{(n-1)(q-1)}}\mathrm d\mathcal H^n_g&\leq\alpha_0^{-\frac1{q-1}}\Gamma\left(\dfrac{q}{q-1}\right)\int_\Sigma\exp_{\frac{q}{q-1}}\left(\alpha_0|u|^{\frac{n}{n-1}}\right)\mathrm d\mathcal H^n_g\nonumber\\
&\leq\alpha_0^{-\frac1{q-1}}\Gamma\left(\dfrac{q}{q-1}\right)\int_\Sigma\exp_{n}\left(\alpha_0|u|^{\frac{n}{n-1}}\right)\mathrm d\mathcal H^n_g\nonumber\\
&\leq C\alpha_0^{-\frac1{q-1}}\Gamma\left(\dfrac{q}{q-1}\right)\|u\|^n_{L^n(\Sigma)},\label{eqpp6}
\end{align}
where we also used the definition of $\exp_p$ for real $p$ (see \cite[pp~4-5]{MR4819618} for details). Note that rescaling by $\alpha_0$ in this way is necessary: the constant $\alpha_\Sigma$ may be smaller than $1$, so Corollary \ref{cor1} cannot be applied to $\exp_n(|u|^{\frac{n}{n-1}})$ directly. The resulting factor $\alpha_0^{-\frac{1}{q-1}}$ is harmless because it enters \eqref{eqpp4} raised to the power $\frac{q-1}{q}$, and
\begin{equation}\label{alpha0bound}
\left(\alpha_0^{-\frac{1}{q-1}}\right)^{\frac{q-1}{q}}=\alpha_0^{-\frac1q}\leq\alpha_0^{-1},\quad\forall q\geq1,
\end{equation}
since $\alpha_0\leq1$. Combining \eqref{eqpp4}, \eqref{eqpp5}, \eqref{eqpp6}, and \eqref{alpha0bound}, and applying Theorem \ref{theoEG}, we deduce that
\begin{align*}
\int_{\{|u|\geq1\}}\exp_n\left(\alpha_\Sigma|u|^{\frac{n}{n-1}}\right)\mathrm d\mathcal H^n_g&\leq 2^{\frac{n}{(n-1)q}}\left(qC\|q^{\frac{n-1}{n}}u\|^n_{L^n(\Sigma)}\right)^{\frac1q}\left[C\Gamma\left(\frac{q}{q-1}\right)\|u\|^n_{L^n(\Sigma)}\right]^{\frac{q-1}q}\\
&\leq Cq^{\frac{n}{q}}\Gamma\left(\dfrac{q}{q-1}\right)^{\frac{q-1}q}\|u\|^n_{L^n(\Sigma)}=C\dfrac{n-1}{\tau}q^{\frac{n}{q}}\Gamma\left(\dfrac{q}{q-1}\right)^{\frac{q-1}q}\dfrac{q-1}q,
\end{align*}

where in the last step we used that the definition of $q$ gives $\frac{q-1}{q}=\frac{\tau}{n-1}\|u\|^n_{L^n(\Sigma)}$, i.e.\ $\|u\|^n_{L^n(\Sigma)}=\frac{n-1}{\tau}\frac{q-1}{q}$.

Therefore, the proof of \eqref{eqpp2} is complete once we show that the function
\begin{equation}\label{eqpp9}
F(q):=q^{\frac{n}{q}}\Gamma\left(\dfrac{q}{q-1}\right)^{\frac{q-1}q}\dfrac{q-1}q
\end{equation}
is bounded for $q\in(1,\infty)$. This is stronger than needed, since $\|u\|^n_{L^n(\Sigma)}<\frac{n-1}{\tau n}$ already implies $q<\frac{n}{n-1}$. It is not hard to see that $F(q)\overset{q\to\infty}\longrightarrow1$. To analyze the behavior as $q\to1$, we set $q'=\frac{q}{q-1}$ and apply Stirling's formula for the Gamma function to obtain
\begin{equation*}
\Gamma(q')=\sqrt{2\pi}(q')^{q'-\frac12}e^{-q'}(1+o(1)),\quad\mbox{as }q'\to\infty.
\end{equation*}
Consequently,
\begin{equation*}
F(q)=q^{\frac{n}q}q'e^{-1}(2\pi)^{\frac{1}{2q'}}(q')^{-\frac{1}{2q'}}(1+o(1))\dfrac{1}{q'}\overset{q\to1}\longrightarrow e^{-1}.
\end{equation*}
This shows that \eqref{eqpp9} is bounded, and the proof of the corollary is complete.
\end{proof}

Finally, we conclude with the proof of our last main theorem.
\begin{proof}[Proof of Theorem \ref{theosuperc}]
Let us again consider the Moser functions $\varphi_\ell\colon\mathbb R^n\to\mathbb R$ defined in \eqref{MoserRn}. For $\ell$ sufficiently large, let $\psi_\ell$ denote the transport of $\varphi_\ell$ to the manifold, as defined in \eqref{MoserSigma}. We replace $\psi_\ell$ by $\psi_\ell/\|\nabla^\Sigma\psi_\ell\|_{L^n(\Sigma)}$ such that $\|\nabla^\Sigma\psi_\ell\|_{L^n(\Sigma)}=1$. By the same argument as in \eqref{psilleq}, we obtain a relation between $\psi_\ell$ and $\varphi_\ell$. More precisely, for $y\in B_{\frac{1}{\ell}}(x_0)$,
\begin{equation*}
\left(1-\dfrac{C}{\ell^2}\right)^{\frac{n+2}{2n}}\varphi_\ell(\exp_{x_0}^{-1}(y))\leq\psi_\ell(y)\leq\left(1+\dfrac{C}{\ell^2}\right)^{\frac{n+2}{2n}}\varphi_\ell(\exp_{x_0}^{-1}(y)).
\end{equation*}

We first estimate the $\|\psi_\ell\|_{L^n(\Sigma)}$. For $\ell$ large enough, we get
\begin{align*}
\|\psi_\ell\|^n_{L^n(\Sigma)}&=\int_{B_{\frac1{\ell+1}}(x_0)}|\psi_\ell|^n\mathrm d\mathcal H^n_g=\left(1+\dfrac{C}{\ell^2}\right)^{-\frac{n+2}2}\int_{B_{\frac1{\ell+1}}(0)}|\varphi_\ell(x)|^n\sqrt{\det g(x)}\mathrm dx\\
&\leq C\left(\dfrac{\ell}{n}\right)^{n-1}\int_0^{\frac1{\ell+1}}\left|\eta\left(\dfrac{n}{\ell}\log\dfrac{R}{r}\right)\right|^nr^{n-1}\mathrm dr.
\end{align*}
Using that $\eta\equiv1$ on $[0,e^{-\frac{\ell}{n}}/(\ell+1)]$ and $\eta(s)=s-L$ on $[e^{-\frac{\ell}{n}}/(\ell+1),1/(\ell+1)]$, we have
\begin{align*}
\|\psi_\ell\|^n_{L^n(\Sigma)}&\leq C\ell^{n-1}\int_0^{\frac{e^{-\frac{\ell}{n}}}{\ell+1}}r^{n-1}\mathrm dr+C\ell^{n-1}\int_{\frac{e^{-\frac{\ell}{n}}}{\ell+1}}^{\frac1{\ell+1}}\left|\frac{n}\ell\log\frac{R}{r}-\frac{n}{\ell}\log\left[R(\ell+1)\right]\right|^nr^{n-1}\mathrm dr\\
&=C\ell^{n-1}\dfrac{e^{-\ell}}{(\ell+1)^n}+\dfrac{C}{\ell}\int_{\frac{e^{-\frac{\ell}{n}}}{\ell+1}}^{\frac1{\ell+1}}|\log(r(\ell+1))|^nr^{n-1}\mathrm dr\\
&\leq C\dfrac{e^{-\ell}}{\ell}+\dfrac{C}{\ell^{n+1}}\int_0^{\frac{\ell}{n}}s^{n}e^{-ns}\mathrm ds\leq \dfrac{C}{\ell^{n+1}}
\end{align*}
for $\ell$ large enough. Given $\alpha\geq n\omega_{n-1}^{\frac1{n-1}}$ and $q\geq0$, by the estimate above, we obtain
\begin{align*}
I_{\alpha,q}(\ell)&:=\dfrac{1}{\|\psi_\ell\|_{L^n(\Sigma)}^n}\int_\Sigma\dfrac{\exp_n\left(\alpha|\psi_\ell|^{\frac{n}{n-1}}\right)}{(1+|\psi_\ell|)^q}\mathrm d\mathcal H^n_g\\
&\geq C\ell^{n+1}\int_{B_{\frac{1}{\ell+1}}(0)}\dfrac{\exp_n\left[\alpha\left(1-\frac{C}{\ell^2}\right)^{\frac{n+2}{2n-2}}|\varphi_\ell|^{\frac{n}{n-1}}\right]}{\left(1+\left(1+\frac{C}{\ell^2}\right)^{\frac{n+2}{2n}}|\varphi_\ell|\right)^q}\mathrm dx.
\end{align*}
Since $\eta\equiv1$ in the ball of radius $e^{-\frac{\ell}{n}}/(\ell+1)$ centered at the origin, we have
\begin{align}
I_{\alpha,q}(\ell)&\geq C\ell^{n+1}\dfrac{\exp_n\left[\frac{\alpha\ell}{n}\omega_{n-1}^{-\frac1{n-1}}\left(1-\frac{C}{\ell^2}\right)^{\frac{n+2}{2n-2}}\right]}{\left(1+\left(1+\frac{C}{\ell^2}\right)^{\frac{n+2}{2n}}\left(\frac{\ell}{n}\right)^{\frac{n-1}{n}}\omega_{n-1}^{-\frac{1}{n}}\right)^q}\int_0^{\frac{e^{-\frac{\ell}{n}}}{\ell+1}}r^{n-1}\mathrm dr\nonumber\\
&\geq C\ell^{n+1-q\frac{n-1}{n}}\exp\left[\frac{\alpha\ell}{n}\omega_{n-1}^{-\frac1{n-1}}\left(1-\frac{C}{\ell^2}\right)^{\frac{n+2}{2n-2}}\right]\dfrac{e^{-\ell}}{(\ell+1)^n}\nonumber\\
&\geq C\ell^{1-q\frac{n-1}n}\exp\left[\ell\left(\dfrac{\alpha}{n}\omega_{n-1}^{-\frac{1}{n-1}}\left(1-\frac{C}{\ell^2}\right)^{\frac{n+2}{2n-2}}-1\right)\right].\label{Ito14}
\end{align}
In particular, if we do not insert $1/\|\psi_\ell\|^n_{L^n(\Sigma)}$ into the calculation, we have
\begin{equation*}
\int_\Sigma\dfrac{\exp_n\left(\alpha|\psi_\ell|^{\frac{n}{n-1}}\right)}{(1+|\psi_\ell|)^q}\mathrm d\mathcal H^n_g\geq C\ell^{-n-q\frac{n-1}n}\exp\left[\ell\left(\dfrac{\alpha}{n}\omega_{n-1}^{-\frac{1}{n-1}}\left(1-\frac{C}{\ell^2}\right)^{\frac{n+2}{2n-2}}-1\right)\right].
\end{equation*}
The limit \eqref{eqsuc2} follows immediately from the above inequality, since the exponential growth dominates any polynomial growth. Moreover, since $1-q\frac{n-1}{n}>0$, the limit \eqref{eqsuc1} also follows from \eqref{Ito14} provided that
\begin{equation}\label{sufftoprove}
\ell\left[\left(1-\frac{C}{\ell^2}\right)^{\frac{n+2}{2n-2}}-1\right]\overset{\ell\to\infty}\longrightarrow0.
\end{equation}
To verify \eqref{sufftoprove}, we use the first-order Taylor expansion of the function $t\mapsto t^{\frac{n+2}{2n-2}}$ around $t=1$:
\begin{equation*}
\left(1-\frac{C}{\ell^2}\right)^{\frac{n+2}{2n-2}}=1-\frac{n+2}{2n-2}\frac{C}{\ell^2}+o\left(\frac1{\ell^2}\right).
\end{equation*}
Thus,
\begin{equation*}
\ell\left[\left(1-\frac{C}{\ell^2}\right)^{\frac{n+2}{2n-2}}-1\right]=-\frac{n+2}{2n-2}\frac{C}{\ell}+o\left(\frac1{\ell}\right)\overset{\ell\to\infty}\longrightarrow0.
\end{equation*}
This proves \eqref{sufftoprove} and completes the proof of the theorem.
\end{proof}

\smallskip

\noindent\textbf{Funding.} A. X. Do and G. Lu were partially supported by grants from the Simons Foundation. N. Lam was partially supported by an NSERC Discovery Grant. R. Ponciano acknowledges partial support from São Paulo Research Foundation (FAPESP) grants 2023/07697-9 and 2025/07027-9.

\end{document}